\documentclass[10pt]{article}
\usepackage{amssymb}
\usepackage{amsfonts}
\usepackage{graphicx}
\usepackage{amsmath}
\usepackage{float,color,ulem}
\usepackage{epsf,epsfig,subfigure}
\usepackage{float,color,ulem}

\newfloat{figure}{H}{lof}
\newfloat{table}{H}{lot}
\floatname{figure}{\figurename}
\floatname{table}{\tablename}
\newtheorem{theorem}{Theorem}[section]

\newtheorem{claim}[theorem]{Claim}

\newtheorem{definition}[theorem]{Definition}

\newtheorem{lemma}[theorem]{Lemma}

\newtheorem{proposition}[theorem]{Proposition}
\newtheorem{remark}[theorem]{Remark}

\newenvironment{proof}[1][Proof]{\noindent\textit{#1.} }{\hfill \rule{0.5em}{0.5em}}
\numberwithin{equation}{section}

\begin{document}

\title{\textbf{Persistence of exponential trichotomy for linear operators:
A\ Lyapunov-Perron approach}}
\author{\textsc{\ A. Ducrot, P. Magal, O. Seydi} \\
{\small \textit{Institut de Mathematiques de Bordeaux, UMR CNRS 5251, }} \\
{\small \textit{\ Universite Bordeaux Segalen, 3ter place de la Victoire,
33000 Bordeaux, France}} }
\date{}
\maketitle

\begin{abstract}
In this article we revisit the perturbation of exponential trichotomy of
linear difference equation in Banach space by using a Perron-Lyapunov \cite%
{Perron} fixed point formulation for the perturbed evolution operator. This
approach allows us to directly re-construct the perturbed semiflow without
using shift spectrum arguments. These arguments are presented to the case of
linear autonomous discrete time dynamical system. This result is then
coupled to Howland semigroup procedure to obtain the persistence of
exponential trichotomy for non-autonomous difference equations as well as
for linear random difference equations in Banach spaces.
\end{abstract}

\section{Introduction}

Let $A\in \mathcal{L}\left( X\right) $ be a bounded linear operator on a
Banach space $(X,\left\Vert .\right\Vert ).\ $Recall that the spectral
radius of $A$ is defined by 
\begin{equation*}
r\left( A\right) :=\lim_{n\rightarrow +\infty }\left\Vert A^{n}\right\Vert _{%
\mathcal{L}\left( X\right) }^{1/n}.
\end{equation*}%
Assume that $A$ has a \textbf{state space decomposition}, whenever $A$ is
regarded as the following discrete time dynamical system 
\begin{equation}
\left\{ 
\begin{array}{l}
x_{n+1}=Ax_{n},\text{ for }n\in \mathbb{N}, \\ 
x_{0}=x\in X.%
\end{array}%
\right.   \label{1.1}
\end{equation}%
Namely, we can find three closed subspaces $X_{s}$ the \textbf{stable
subspace}, $X_{c}$ the \textbf{central subspace}, and $X_{u}$ the \textbf{%
unstable subspace} such that 
\begin{equation*}
X=X_{s}\oplus X_{c}\oplus X_{u}\text{ and }A\left( X_{k}\right) \subset
X_{k},,\forall k=s,c,u.
\end{equation*}%
Moreover if we define for each $k=s,c,u$ $A_{k}\in \mathcal{L}\left(
X_{k}\right) $ the part of $A$ in $X_{k}$ (i.e. $A_{k}x=Ax,\forall x\in X_{k}
$). Then there exists a constant $\alpha \in \left( 0,1\right) $ such that 
\begin{equation*}
r\left( A_{s}\right) \leq \alpha <1
\end{equation*}%
the linear operator $A_{u}$ on $X_{u}$ is invertible and 
\begin{equation*}
r\left( A_{u}^{-1}\right) \leq \alpha <1
\end{equation*}%
and the operator $A_{c}$ on $X_{c}$ is invertible and 
\begin{equation*}
r\left( A_{c}\right) <\alpha ^{-1}\text{ and }r\left( A_{c}^{-1}\right)
<\alpha ^{-1}.
\end{equation*}%
We summarize the notion of state space decomposition into the following
definition.\ In the context of linear dynamical system (or linear
skew-product semiflow) this notion also corresponds to the notion of
exponential trichotomy. The following definition corresponds to the one
introduced by Hale and Lin in \cite{Hale}.

\begin{definition}
\label{DE1.0}Let $A\in \mathcal{L}\left( X\right) $ be a bounded linear
operator on a Banach space $(X,\left\Vert .\right\Vert ).\ $We will say that 
$A$ has an \textbf{exponential trichotomy} (or $A$ is \textbf{exponentially
trichotomic}) if there exist three bounded linear projectors $\Pi _{s},\Pi
_{c},\Pi _{u}\in \mathcal{L}\left( X\right) $ such that%
\begin{equation*}
X=X_{s}\oplus X_{c}\oplus X_{u},
\end{equation*}%
and 
\begin{equation*}
A\left( X_{k}\right) \subset X_{k},,\forall k=s,c,u,
\end{equation*}%
where $X_{k}:=\Pi _{k}\left( X\right)$, for $k=s,c,u$, and%
\begin{equation*}
X_{c}\oplus X_{u}=(I-\Pi _{s})(X)\text{, }X_{s}\oplus X_{u}=(I-\Pi _{c})(X)%
\text{ and }X_{s}\oplus X_{c}=(I-\Pi _{u})(X).
\end{equation*}%
Moreover we assume that there exists a constant $\alpha \in \left(
0,1\right) $ satisfying the following properties:

\begin{itemize}
\item[(i)] Let $A_{s}\in \mathcal{L}\left( X_{s}\right) $ be the part of $A$
in $X_{s}$ (i.e. $A_{s}(x)=A(x),\forall x\in X_{s}$) we assume that $r\left(
A_{s}\right) \leq \alpha$ ;

\item[(ii)] Let $A_{u}\in \mathcal{L}\left( X_{u}\right) $ be the part of $A$
in $X_{u}$ (i.e. $A_{u}(x)=A(x),\forall x\in X_{u}$) we assume that $A_{u}$
is invertible and $r\left( A_{u}^{-1}\right) \leq \alpha$;

\item[(iii)] Let $A_{c}\in \mathcal{L}\left( X_{c}\right) $ be the part of $%
A $ in $X_{c}$ (i.e. $A_{c}(x)=A(x),\forall x\in X_{c}$) we assume that $%
A_{c}$ is invertible and $r\left( A_{c}\right) <\alpha ^{-1}$ and $r\left(
A_{c}^{-1}\right) <\alpha ^{-1}$.
\end{itemize}
\end{definition}

Let $A:D(A)\subset X\rightarrow X$ is a linear operator on a Banach space $X$%
. Let $Y\subset X$ is a subspace of $X.$ Recall that $A_{Y}:D\left(
A_{Y}\right) \subset Y\rightarrow Y$ the part of $A$ in $Y$ is defined by%
\begin{equation*}
D\left( A_{Y}\right) :=\left\{ x\in D(A)\cap Y:Ax\in Y\right\}\text{ and }%
A_{Y}x=Ax,\;\forall x\in D\left( A_{Y}\right) .
\end{equation*}%
Note that in Definition \ref{DE1.0} only the forward information are used on
the stable part $X_{s}$, forward and backward for the central part $X_{c}$
while only forward information are necessary on the unstable part $X_{u}$.
This remark motivates the following definition of exponential trichotomy for
unbounded linear operator operator that will be used throughout this work.

\begin{definition}
\label{DE1.1}Let $A:D(A)\subset X\rightarrow X$ be a closed linear operator
on a Banach space $(X,\left\Vert .\right\Vert ).\ $We will say that $A$ has
an {\textbf{e}xponential trichotomy} (or $A$ is \textbf{exponentially
trichotomic}) if there exist three bounded linear projectors $\Pi _{s},\Pi
_{c},\Pi _{u}\in \mathcal{L}\left( X\right) $ such that%
\begin{equation}
X=X_{s}\oplus X_{c}\oplus X_{u},\text{ where }X_{k}:=\Pi _{k}\left( X\right)
,\forall k=s,c,u,  \label{1.2}
\end{equation}
and%
\begin{equation*}
X_{c}\oplus X_{u}=(I-\Pi _{s})(X)\text{, }X_{s}\oplus X_{u}=(I-\Pi _{c})(X)%
\text{ and }X_{s}\oplus X_{c}=(I-\Pi _{u})(X).
\end{equation*}%
Moreover we assume that 
\begin{equation*}
D(A)=X_{s}\oplus X_{c}\oplus \left( D(A)\cap X_{u}\right) .
\end{equation*}%
and 
\begin{equation}
A\left( D(A)\cap X_{k}\right) \subset X_{k},,\forall k=s,c,u.  \label{1.3}
\end{equation}%
Furthermore we assume that there exists a constant $\alpha \in \left(
0,1\right) $ satisfying the following properties:

\begin{itemize}
\item[(i)] Let $A_{s}\in \mathcal{L}(X_{s})$ be the part of $A$ in $X_{s}$,
we assume that%
\begin{equation}
r\left( A_{s}\right) \leq \alpha ;  \label{1.4}
\end{equation}

\item[(ii)] Let $A_{u}:D(A_{u})\subset X_{u}\rightarrow X_{u}$ be the part
of $A$ in $X_{u}$,\ we assume that $A_{u}$ is invertible and 
\begin{equation}
r\left( A_{u}^{-1}\right) \leq \alpha ;  \label{1.5}
\end{equation}

\item[(iii)] Let $A_{c}\in \mathcal{L}(X_{c})$ be the part of $A$ in $X_{c},$%
we assume that $A_{c}$ is invertible and 
\begin{equation}
r\left( A_{c}\right) <\alpha ^{-1}\text{ and }r\left( A_{c}^{-1}\right)
<\alpha ^{-1}\text{.}  \label{1.6}
\end{equation}
\end{itemize}
\end{definition}

\begin{remark}
\label{RE1.2}The above properties (\ref{1.5})-(\ref{1.7}) are also
equivalent to say that there exist three constants $\kappa \geq 1$ and $%
0<\rho _{0}<\rho $ such that 
\begin{equation}
\left\Vert A_{c}^{n}\right\Vert _{\mathcal{L}\left( X_{c}\right) }\leq
\kappa e^{\rho _{0}\left\vert n\right\vert },\forall n\in \mathbb{Z},
\label{1.7}
\end{equation}%
\begin{equation}
\left\Vert A_{s}^{n}\right\Vert _{\mathcal{L}\left( X_{s}\right) }\leq
\kappa e^{-\rho n},\forall n\in \mathbb{N},  \label{1.8}
\end{equation}%
$0\in \rho \left( A_{u}\right) $ and 
\begin{equation}
\left\Vert A_{u}^{-n}\right\Vert _{\mathcal{L}\left( X_{u}\right) }\leq
\kappa e^{-\rho n},\forall n\in \mathbb{N}.  \label{1.9}
\end{equation}%
In the sequel, the above estimates will be referred as exponential
trichotomy with exponents $\rho _{0}<\rho $, constant $\kappa $ and
associated to the projectors $\left\{ \Pi ^{\alpha }\right\} _{\alpha
=s,c,u} $.
\end{remark}

\begin{remark}
Since the linear operator $A$ is assumed to be closed, by the closed graph
theorem, Definition \ref{DE1.1} coincides with Definition \ref{DE1.0} if and
only if $D(A_{u})=X_{u}$.
\end{remark}

By using the definition of exponential trichotomy we may also define the
notion of exponential dichotomy.\ 

\begin{definition}
\label{DE1.3}Let $A:D(A)\subset X\rightarrow X$ be a closed linear operator
on a Banach space $(X,\left\Vert .\right\Vert ).\ $We will say that $A$ has
an \textbf{exponential dichotomy} if $A$ has an exponential trichotomy with $%
X_{c}=\left\{ 0\right\} .$
\end{definition}

The aim of this paper is to study the persistence of exponential trichotomy
(according to Definition \ref{DE1.1}) under small bounded additive
perturbation. Before going to our main result and application to
non-autonomous problem, let us recall that exponential dichotomy (trichotomy
and more generally invariant exponential splitting) is a basic tool to study
stability for non-autonomous dynamical systems (see for instance \cite%
{Coppel, Barreira-Valls-2008, Potzsche} and the references therein). It is
also a powerful ingredient to construct suitable invariant manifolds for
non-linear problems (see \cite{Chicone-Latushkin-1997, Chow-Liu-Yi-2000(a),
Barreira-Valls-2008} and the references therein). In the last decades a lot
of attention and progresses have been made to understand invariant splitting
for non-autonomous linear dynamical systems (continuous time as well as
linear difference equations) as well as their persistence under small
perturbations. We refer for instance to \cite{Sacker-Sell-1974,
Sacker-Sell-1976, Sacker-Sell-1978, Sacker-Sell-1994, Henry, Hale, Popescu,
Latushkin-Schnaubelt, Chow-Leiva-1995, Chow-Leiva-1996} and the references
cited therein.

Let us also mention the notion of non-uniform dichotomy and trichotomy in
which the boundedness of the projectors is relaxed allowing unbounded linear
projectors (see for instance \cite{Barreira-Valls-2007, Barreira-Valls-2008,
Barreira-Valls-2009, Barreira-Valls-2010} for non-autonomous dynamical
system and \cite{Zhou-Lu-Zhang} for random linear difference equations).

The persistence of exponential splitting under small perturbation is also an
important problem with several applications in dynamical system such as
shadowing properties. We refer to Palmer \cite{Palmer-1987, Palmer-2011} and
the reference therein.

Finally we would like to compare our definition of exponential splitting
with the one recently considered by Potzsche in his monograph (see Chapter 3
in \cite{Potzsche}). In the homogeneous case, Potzsche considers exponential
splitting for a pair of linear operator $(A,B)\in \mathcal{L}(X,Y)$ where $X$
and $Y$ denote two Banach spaces. Note that $X$ can be different from $Y$ so
that this framework allows to applies to closed linear operators. Let us
recall that when $(A,B)\in \mathcal{L}(X,Y)$ one may consider the
corresponding ($Y-$valued) linear difference equation on $X$ defined as 
\begin{equation*}
Bx_{k+1}=Ax_k,\;\;k\in\mathbb{Z}.
\end{equation*}
We now recall the definition of exponential dichotomy used by Potzsche in 
\cite{Potzsche}:

\begin{definition}
\label{DEF-Pot} An operator pair $(A,B)\in \mathcal{L}(X,Y)^{2}$ acting
between two Banach spaces $X$ and $Y$ is said to have an exponential
dichotomy if there exist $\kappa >0$, $\rho >0$ and two orthogonal and
complementary projectors $\Pi _{s},\Pi _{u}\in \mathcal{L}(X)$ such that, by
setting $X_{k}=\Pi _{k}\left( X\right) $ for $k=s,u$%
\begin{equation*}
X=X_{s}\oplus X_{u}
\end{equation*}
\begin{equation*}
\begin{split}
& \ker B|_{X_{s}}=\{0\},\;R\left( A\Pi _{s}\right) \subset R\left( B\Pi
_{s}\right) =:Y_{s} \\
& \ker A_{|X_{u}}=\{0\},\;\;R(B\Pi _{u})\subset R(A\Pi _{u})=:Y_{u}
\end{split}%
\end{equation*}%
and $\Phi _{s}:=\left( B_{|Y_{s}}^{-1}A\right) _{|X_{s}}\in \mathcal{L}%
(X_{s})$ and $\Phi _{u}:=\left( A_{|Y_{u}}^{-1}B\right) _{|X_{u}}\in 
\mathcal{L}(X_{u})$ while 
\begin{equation*}
\Vert \Phi _{s}^{n}\Vert _{\mathcal{L}(X_{s})}\leq \kappa e^{-n\rho },\;%
\text{ and }\Vert \Phi _{u}^{n}\Vert _{\mathcal{L}(X_{s})}\leq \kappa
e^{-n\rho }\;\;\forall n\geq 0.
\end{equation*}
\end{definition}

Note that when $A:D(A)\subset X\rightarrow X$ is a closed linear operator,
the application of this theory to the pair $(A,J)\in \mathcal{L}(D(A),X)$
(where $J:D(A)\rightarrow X$ denotes the canonical embedding from $D\left(
A\right) $ into $X$ i.e.\ $J(x)=x,\forall x\in D\left( A\right) $) would
lead us to a splitting of the Banach space $D(A)$ (endowed with the graph
norm). Let us also notice that when $A\in \mathcal{L}(X)$ has an exponential
dichotomy (according to Definition \ref{DE1.0}) then the pair $(A,I_{X})\in 
\mathcal{L}(X)^{2}$ has an exponential dichotomy in the above sense
(Definition \ref{DEF-Pot}). In the same way when $A\in \mathcal{L}(X)$ has
an exponential trichotomy (according to Definition \ref{DE1.0}) then the
pair $(A,I_{X})\in \mathcal{L}(X)^{2}$ has $3-$ exponential invariant
splitting in the sense of Potzsche \cite[Definition 3.4.12 p.135]{Potzsche}.

Now let $A:D(A)\subset X\rightarrow X$ be an exponential dichotomic closed
linear operator with parameter $\kappa >0$ $\rho >0$ and projectors $\Pi
_{k} $ $k=s,u$ (see Definition \ref{DE1.1}). Consider the linear operator $%
\widehat{B}\in \mathcal{L}(X)$ defined by $\widehat{B}=\left( A_{u}^{-1}\Pi
_{u}+\Pi _{s}\right) \in \mathcal{L}(X)$. Then by applying $\widehat{B}$ on
the left side of (\ref{1.1}), the linear difference equation $%
x_{n+1}=Ax_{n} $ becomes 
\begin{equation*}
\widehat{B}x_{n+1}=\widehat{B}Ax_{n}
\end{equation*}%
or equivalently 
\begin{equation*}
\widehat{B}x_{n+1}=\widehat{A}x_{n}
\end{equation*}%
where $\widehat{A}:=\Pi _{u}+A_{s}\Pi _{s}\in \mathcal{L}(X)$ is a bounded
extension of $\widehat{B}A:D(A)\subset X\rightarrow X$ (the unique bounded
extension if $A$ is densely defined). In order to deal with the above linear
difference equation, one may consider the operator pair $\left( \widehat{A},%
\widehat{B}\right) $. Note that $\ker \widehat{B}=\{0\}$ and $\widehat{B}%
^{-1}:D\left( \widehat{B}^{-1}\right) \subset X\rightarrow X$ is the closed
linear operator defined by 
\begin{equation*}
D\left( \widehat{B}^{-1}\right) =R(\widehat{B})=X_{s}\oplus D(A)\cap
X_{u}=D(A)\text{ and }\widehat{B}^{-1}=A_{u}\Pi _{u}+\Pi _{s}.
\end{equation*}%
Then it is easy to check that if $A:D(A)\subset X\rightarrow X$ has an
exponential dichotomy (see Definition \ref{DE1.1}) then the pair $\left( 
\widehat{B},\widehat{A}\right) \in \mathcal{L}(X)^{2}$ has an exponential
dichotomy according to Definition \ref{DEF-Pot}.

One can therefore try to use this operator pair framework to study the
persistence of exponential dichotomy provided by the extended Definition \ref%
{DE1.1}. Let recall that Potzsche derived in his monograph a general
roughness result using the operator pair framework (see Theorem 3.6.5
p.165). Let $\left( A,B\right) \in \mathcal{L}(X,Y)^{2}$ be an exponential
dichotomy operator pair and let $\left( \overline{A},\overline{B}\right) \in 
\mathcal{L}(X,Y)^{2}$ be a given perturbation. Then if 
\begin{equation*}
R(A)\subset R(B),\;R\left( \overline{A}\right) \subset R(B)\;\text{and }%
R\left( \overline{B}\right) \subset R(B)
\end{equation*}
then under suitable smallness assumptions then the operator pair $\left( A+%
\overline{A},B+\overline{B}\right) $ has an exponential dichotomy.\newline

\noindent Consider an exponentially dichotomic closed linear operator $%
A:D(A)\subset X\rightarrow X$ as well as a small perturbation $C\in \mathcal{%
L}(X)$. Using the above transformation the linear difference equation $%
x_{k+1}=(A+C)x_{k}$ rewrites as studying the invariant splitting for the
operator pair $\left( \widehat{B},\widehat{A}+\widehat{B}C\right) \in 
\mathcal{L}(X)^{2}$. In that context, note the compatibility condition $%
R\left( \widehat{A}\right) \subset R\left( \widehat{B}\right) $ re-writes as 
$X_{u}\oplus R\left( A_{s}\right) \subset \left( D(A)\cap X_{u}\right)
\oplus X_{s}$ that is true if and only if $D(A)\cap X_{u}=X_{u}$, that is $%
D(A)=X$ and $A\in \mathcal{L}(X)$. Here since $A$ is closed the closed graph
theorem implies that $A$ is bounded.\ 

As a consequence, the general persistence results of Potzsche in \cite%
{Potzsche} cannot directly apply to study the persistence of the splitting
for the class of linear unbounded operators.

In this work we propose to revisit the problem of persistence of exponential
trichotomy for the class of operators described in Definition \ref{DE1.1} by
dealing with a direct proof based on Perron-Lyapunov fixed point argument
for the perturbed semiflows and projectors. More specifically if $B\in 
\mathcal{L}\left( X\right) $ (with $\left\Vert B\right\Vert _{\mathcal{L}%
\left( X\right) }$ small enough) we aim at investigating the persistence of
such the state space decomposition for a small bounded linear perturbation
of an exponentially trichotomic closed linear operator $A:D(A)\subset X\to X$%
.

The main result of the manuscript is the following theorem.

\begin{theorem}
\label{TH1.4}\textbf{(Perturbation) }Let $A:D(A)\subset X\to X$ be a closed
linear operator on a Banach space $X,$ and assume that $A$ has exponential
trichotomy with exponents $\rho_0<\rho$, constant $\kappa$ and associated to
the projectors $\left\{\Pi^\alpha\right\}_{\alpha=s,c,u}$ (see Remark \ref%
{RE1.2}).\ Then for each $B\in \mathcal{L}\left( X\right) $ with $\left\Vert
B\right\Vert _{\mathcal{L}\left( X\right) }$ small enough the closed linear
operator $\left( A+B\right):D(A)\subset X\to X$ has an exponential
trichotomy,\ which corresponds to the following state space decomposition 
\begin{equation*}
X=\widehat{X}_{s}\oplus \widehat{X}_{c}\oplus \widehat{X}_{u},
\end{equation*}%
and which corresponds to the bounded linear projectors $\widehat{\Pi }_{s},%
\widehat{\Pi }_{c},\widehat{\Pi }_{u}\in \mathcal{L}\left( X\right) $
satisfying 
\begin{equation*}
\widehat{X}_{k}:=\widehat{\Pi }_{k}\left( X\right) ,\forall k=s,c,u,
\end{equation*}%
and%
\begin{equation*}
\widehat{X}_{c}\oplus \widehat{X}_{u}=(I-\widehat{\Pi }_{s})(X)\text{, }%
\widehat{X}_{s}\oplus \widehat{X}_{u}=(I-\widehat{\Pi }_{c})(X)\text{ and }%
\widehat{X}_{s}\oplus \widehat{X}_{c}=(I-\widehat{\Pi }_{u})(X).
\end{equation*}%
Moreover precisely, let three constants $\widehat{\rho }_{0},\widehat{\rho }%
\in \left( 0,+\infty \right)$ and $\widehat\kappa$ be given such that 
\begin{equation*}
0<\rho _{0}<\widehat{\rho }_{0}<\widehat{\rho }<\rho \text{ and } \widehat{%
\kappa }>\kappa .
\end{equation*}%
There exists $\delta _{0}=\delta _{0}\left( \rho _{0},\widehat{\rho }_{0},%
\widehat{\rho },\rho ,\kappa ,\widehat{\kappa }\right) \in \left( 0,\sqrt{2}%
-1\right) $ such that for each $\delta \in \left( 0,\frac{\delta _{0}^{2}}{%
\kappa +\delta _{0}}\right) $ if $\left\Vert B\right\Vert _{\mathcal{L}%
\left( X\right) }\leq \delta$, then $\left( A+B\right) $ has an exponential
trichotomy with exponent $\widehat{\rho }_{0}$ and $\widehat{\rho }$ and
with constant $\widehat{\kappa }.\ $

Moreover, the three associated projectors $\widehat{\Pi }_{s},\widehat{\Pi }%
_{c},\widehat{\Pi }_{u}\in \mathcal{L}\left( X\right) $ satisfy%
\begin{equation*}
\left\Vert \widehat{\Pi }_{k}-\Pi _{k}\right\Vert _{\mathcal{L}\left(
X\right) }<\frac{\kappa \delta }{\delta _{0}-\delta }\leq \delta _{0}<\sqrt{2%
}-1,\forall k=s,c,u,
\end{equation*}%
and as a consequence the subspace $\widehat{X}_{k}:=\widehat{\Pi }_{k}\left(
X\right) $ is isomorphic to the subspace $X_{k}=\Pi _{k}\left( X\right) .$

Furthermore the following estimates hold true for each $n\in \mathbb{N},$%
\begin{equation*}
\left\Vert \left( A+B\right) _{s}^{n}\widehat{\Pi }_{s}-A_{s}^{n}\Pi
_{s}\right\Vert _{\mathcal{L}\left( X\right) }\leq \frac{\kappa \delta }{%
\delta _{0}-\delta }e^{-\widehat{\rho }n},
\end{equation*}%
\begin{equation*}
\left\Vert \left( A+B\right) _{u}^{-n}\widehat{\Pi }_{u}-A_{u}^{-n}\Pi
_{u}\right\Vert _{\mathcal{L}\left( X\right) }\leq \frac{\kappa \delta }{%
\delta _{0}-\delta }e^{-\widehat{\rho }n},
\end{equation*}%
and for each $n\in \mathbb{Z}$%
\begin{equation*}
\left\Vert \left( A+B\right) _{c}^{n}\widehat{\Pi }_{c}-A_{c}^{n}\Pi
_{c}\right\Vert _{\mathcal{L}\left( X\right) }\leq \frac{\kappa \delta }{%
\delta _{0}-\delta }e^{\widehat{\rho }_{0}\left\vert n\right\vert }.
\end{equation*}
\end{theorem}

In case of bounded linear operator the above result is a particular case of
the result proved by Potzsche in \cite{Potzsche} and by Pliss and Sell \cite%
{Pliss-Sell} using perturbation of exponential dichotomy for linear skew
product semiflow and shifted operators. For the class of unbounded linear
operator we consider in this work this result is new.

In addition, the above result has some consequence for non-autonomous
discrete time linear equations by using Howland semigroup procedure to
re-formulate such problems as autonomous systems.

In the next subsection we will state some consequences of Theorem \ref{TH1.4}%
. Section 2 is devoted to the proof of Theorem \ref{TH1.4}. Section 3 is
concerned with the application of Theorem \ref{TH1.4} for non-autonomous
dynamical system (see Theorem \ref{TH1.8}). We also refer to Seydi \cite%
{Seydi} for further application in the context random dynamical systems and
shadowing of normally hyperbolic dynamics.\ 

\section{Consequences of Theorem \protect\ref{TH1.4} \ for discrete time
non-autonomous dynamical system}

As mentioned above, exponential trichotomy or dichotomy play an important
role in the study of the asymptotic behaviour of non-autonomous dynamical
systems. Roughly speaking exponential trichotomy generalizes the usual
spectral theory of linear semigroups to linear evolution operators. It
ensures an invariant state space decomposition at each time into three
sub-spaces: a stable, an unstable and central space in which the the
evolution operator has different exponential behaviours. Let $\mathbf{A}%
=\left\{ A_{n}\right\} _{n\in \mathbb{Z}}:\mathbb{Z}\rightarrow \mathcal{L}%
(Y)$ be a given sequence of bounded linear operators on the Banach space $%
(Y,\Vert \;\Vert )$. Consider the linear non-autonomous difference equation 
\begin{equation}
x(n+1)=A_{n}x(n),\text{ for }n\geq m,\;\; x(m)=x_{m}\in Y.  \label{1.11}
\end{equation}%
Let us introduce the discrete evolution semigroup associated to $\mathbf{A}$
defined as the $2-$parameters linear operator on $\Delta _{+}:=\left\{
(n,m)\in \mathbb{Z}^{2}:\;n\geq m\right\} $ by 
\begin{equation*}
U_{\mathbf{A}}\left( n,m\right) := A_{n-1}...A_{m}\text{, \ if }n>m, \text{
and } I_{Y},\text{ if }n=m,
\end{equation*}%
wherein $I_{Y}$ denotes the identity operator in $Y$. In the following we
will always use the notation $n\geq m$ as well as $U_{\mathbf{A}}\left(
n,m\right) $ for the evolution semigroup. Whenever $U_{\mathbf{A}}\left(
m,n\right) $ is considered this will mean that $U_{\mathbf{A}}\left(
n,m\right) $ is invertible and 
\begin{equation*}
U_{\mathbf{A}}\left( m,n\right) =U_{\mathbf{A}}\left( n,m\right) ^{-1}.
\end{equation*}%
Let us observe that $U_{\mathbf{A}}$ satisfies 
\begin{equation*}
U_{\mathbf{A}}\left( n,k\right) U_{\mathbf{A}}\left( k,m\right) =U_{\mathbf{A%
}}\left( n,m\right) \text{ for each }n\geq k\geq m.
\end{equation*}%
Then let us recall the following definition taken from Hale and Lin \cite%
{Hale}.

\begin{definition}[Exponential trichotomy]
\label{DE1.6} Let $\mathbf{A}=\left\{ A_{n}\right\} _{n\in \mathbb{Z}}:%
\mathbb{Z}\rightarrow \mathcal{L}(Y)$ be given. Then $U_{\mathbf{A}}$ has an 
\textbf{exponential trichotomy} (or $\mathbf{A}$ is \textbf{exponentially
trichotomic}) on $\mathbb{Z}$ with constant $\kappa $, exponents $0<\rho
_{0}<\rho $ if there exist three families of projectors $\mathbf{\Pi }%
^{\alpha }=\left\{ \Pi _{n}^{\alpha }\right\} _{n\in \mathbb{Z}}:\mathbb{Z}%
\rightarrow \mathcal{L}(Y)$, with $\alpha =u,s,c$ satisfying the following
properties:

\begin{itemize}
\item[(i)] For all $n\in \mathbb{Z}$ and $\alpha ,\beta \in \left\{
u,s,c\right\} $ we have%
\begin{equation*}
\Pi _{n}^{\alpha }\Pi _{n}^{\beta }=0,\ \text{if }\alpha \neq \beta ,\text{
and } \Pi _{n}^{s}+\Pi _{n}^{u}+\Pi _{n}^{c}=I_Y.
\end{equation*}

\item[(ii)] For all $n,m$ $\in \mathbb{Z}$ with $n\geq m$ we have 
\begin{equation*}
U_{\mathbf{A}}^{\alpha }\left( n,m\right) :=\Pi _{n}^{\alpha }U_{\mathbf{A}%
}\left( n,m\right) =U_{\mathbf{A}}\left( n,m\right) \Pi _{m}^{\alpha },\ 
\text{for }\alpha =u,s,c.
\end{equation*}

\item[(iii)] $U_{\mathbf{A}}^{\alpha }\left( n,m\right) $ is invertible from 
$\Pi _{m}^{\alpha }\left( Y\right) $ into $\Pi _{n}^{\alpha }\left( Y\right) 
$ for all $n\geq m$ in $\mathbb{Z},$ $\alpha =u,c$ and its inverse is
denoted by $U_{\mathbf{A}}^{\alpha }\left( m,n\right) :\Pi _{n}^{\alpha
}\left( Y\right) \rightarrow \Pi _{m}^{\alpha }\left( Y\right) $.

\item[(iv)] For each $y\in Y$ we have for all $n,m\in \mathbb{Z}$%
\begin{equation}
\left\Vert U_{\mathbf{A}}^{c}\left( n,m\right) \Pi _{m}^{c}y\right\Vert \leq
\kappa e^{\rho _{0}\left\vert n-m\right\vert }\left\Vert y\right\Vert ,
\label{1.12}
\end{equation}%
and if $n\geq m$%
\begin{equation}
\left\Vert U_{\mathbf{A}}^{s}\left( n,m\right) \Pi _{m}^{s}y\right\Vert \leq
\kappa e^{-\rho \left( n-m\right) }\left\Vert y\right\Vert ,  \label{1.13}
\end{equation}%
\begin{equation}
\left\Vert U_{\mathbf{A}}^{u}\left( m,n\right) \Pi _{n}^{u}y\right\Vert \leq
\kappa e^{-\rho \left( n-m\right) }\left\Vert y\right\Vert .  \label{1.14}
\end{equation}
\end{itemize}
\end{definition}

Let us observe that the operators $U_{\mathbf{A}}^{\alpha }\left( n,p\right)
\in \mathcal{L}\left( Y\right) ,$ for $n\geq p$ in $\mathbb{Z}$ and $\alpha
=u,s,c$ (resp. $U_{\mathbf{A}}^{\alpha }\left( p,n\right) ,$ for $n\geq p$
in $\mathbb{Z}$ and $\alpha =u,c$) inherit the evolution property of $U_{%
\mathbf{A}}$ that reads as 
\begin{equation*}
U_{\mathbf{A}}^{\alpha }\left( n,p\right) U_{\mathbf{A}}^{\alpha }\left(
p,m\right) =U_{\mathbf{A}}^{\alpha }\left( n,m\right) ,\ \forall n\geq p\geq
m\text{ in }\mathbb{Z}\text{ and }\alpha =u,s,c\text{,}
\end{equation*}%
respectively%
\begin{equation*}
U_{\mathbf{A}}^{\alpha }\left( m,p\right) U_{\mathbf{A}}^{\alpha }\left(
p,n\right) =U_{\mathbf{A}}^{\alpha }\left( m,n\right) ,\ \forall n\geq p\geq
m\text{ in }\mathbb{Z}\text{ and }\alpha =u,c\text{.}
\end{equation*}%
Before stating our result, let us notice that since $\Pi _{n}^{\alpha }=U_{%
\mathbf{A}}^{\alpha }\left( n,n\right) ,$ for $\alpha =u,s,c$ and $n\in 
\mathbb{Z},$ property (\textit{iv)} in Definition \ref{DE1.6} implies that
the projectors are uniformly bounded by the constant $\kappa $.

Using Howland's semigroups like procedure (see Chicone and Latushkin \cite%
{Chicone-Latushkin-1999}), as a consequence of Theorem \ref{TH1.4}, we
obtain the following version for non-autonomous dynamical systems.

\begin{theorem}
\label{TH1.8}\textbf{(Perturbation)} Let $\mathbf{A}:\mathbb{Z}\rightarrow 
\mathcal{L}(Y)$ be given such that $U_{\mathbf{A}}$ has an exponential
trichotomy on $\mathbb{Z}$ with constant $\kappa ,$ exponents $0<\rho
_{0}<\rho $ and associated to the three families of projectors $\left\{ 
\mathbf{\Pi }^{\alpha }:\mathbb{Z}\rightarrow \mathcal{L}(Y)\right\}
_{\alpha =s,c,u}$. Let $\rho _{0}<\widehat{\rho }_{0}<\widehat{\rho }<\rho $
and $\widehat{\kappa }>\kappa $ be given. Then there exists $\delta
_{0}:=\delta _{0}\left( \rho _{0},\widehat{\rho }_{0},\widehat{\rho },\rho
,\kappa ,\widehat{\kappa }\right) \in \left( 0,\sqrt{2}-1\right) $ such that
for each $\delta \in \left( 0,\frac{\delta _{0}^{2}}{\kappa +\delta _{0}}%
\right) $ and each $\mathbf{B}:\mathbb{Z}\rightarrow \mathcal{L}(Y)$ with 
\begin{equation*}
\sup_{n\in \mathbb{Z}}\left\Vert B_{n}\right\Vert \leq \delta ,
\end{equation*}%
the evolution semigroup $U_{\mathbf{A}+\mathbf{B}}$ has an exponential
trichotomy on $\mathbb{Z}$ with constant $\widehat{\kappa },$ exponents $%
\widehat{\rho },\ \widehat{\rho }_{0}$ and projectors $\left\{ \widehat{%
\mathbf{\Pi }}^{\alpha }\right\} _{\alpha =s,c,u}$. For each $n\in \mathbb{Z}
$ and $\alpha =u,s,c,$ the spaces $\mathcal{R}\left( \Pi _{n}^{\alpha
}\right) $ and $\mathcal{R}\left( \widehat{\Pi }_{n}^{\alpha }\right) $ are
isomorphic. Moreover the following perturbation estimates hold true: we have
for all $n\geq p,$%
\begin{equation}
\left\Vert {U}_{\mathbf{A}+\mathbf{B}}^{s}\left( n,p\right) -U_{\mathbf{A}%
}^{s}\left( n,p\right) \right\Vert \leq \frac{\kappa \delta }{\delta
_{0}-\delta }e^{-\widehat{\rho }\left( n-p\right) },  \label{1.15}
\end{equation}%
\begin{equation}
\left\Vert {U}_{\mathbf{A}+\mathbf{B}}^{u}\left( p,n\right) -U_{\mathbf{A}%
}^{u}\left( p,n\right) \right\Vert \leq \frac{\kappa \delta }{\delta
_{0}-\delta }e^{-\widehat{\rho }\left( n-p\right) },  \label{1.16}
\end{equation}%
and for all $\left( n,p\right) \in \mathbb{Z}$%
\begin{equation}
\left\Vert {U}_{\mathbf{A}+\mathbf{B}}^{c}\left( n,p\right) -U_{\mathbf{A}%
}^{c}\left( n,p\right) \right\Vert \leq \frac{\kappa \delta }{\delta
_{0}-\delta }e^{\widehat{\rho }_{0}\left\vert n-p\right\vert }\text{.}
\label{1.17}
\end{equation}
\end{theorem}

The proof of the above result will be obtained as a consequence of Theorem %
\ref{TH1.4} by using Howland procedure (we refer to the monograph of Chicone
and Latushkin \cite{Chicone-Latushkin-1999} and the references therein). To
be more precise, let $q\in \lbrack 1,\infty ]$ be given and let us introduce
the Banach space $X=l^{q}(\mathbb{Z},Y)$. Let us consider the closed linear
operator $\mathcal{A}:D\left(\mathcal{A}\right)\subset l^{q}(\mathbb{Z}%
;Y)\rightarrow l^{q}(\mathbb{Z};Y)$ defined by 
\begin{equation}  \label{operateur_A}
\begin{split}
&D\left(\mathcal{A}\right)=\left\{u\in X:\;\left\{A_{k}u\right\}_{k\in%
\mathbb{Z}}\in X\right\}, \\
&\left( \mathcal{A}u\right) _{k}=A_{k-1}u_{k-1},\;\;\forall k\in \mathbb{Z}%
,\;\forall u\in l^{q}(\mathbb{Z};Y).
\end{split}%
\end{equation}%
Then we will show if $\mathbf{A}$ has an exponential trichotomy (according
to Definition \ref{DE1.6}) then the linear operator $\mathcal{A}$ also has
an exponential trichotomy (see Definition \ref{DE1.1}). Note that Theorem %
\ref{TH1.8} is not a new result (see P\"{o}tzsche \cite{Potzsche} and the
reference therein). However since we do not assume that the sequence $%
\mathbf{A}$ is uniformly bounded, the Howland evolution operator is not
bounded and therefore our proof of Theorem \ref{TH1.8} is new.

Note that since operator $\left(\mathcal{A},D\left(\mathcal{A}\right)\right)$%
, one can apply Theorem \ref{TH1.4} with general small perturbation $%
\mathcal{B}=\left(B_{i,j}\right)_{(i,j)\in \mathbb{Z}^2}\in\mathcal{L}(X)$
to obtain the peristence of exponential dichotomy or trichotomy for some
advanced and retarded perturbation of \eqref{1.11} of the form: 
\begin{equation*}
x_{n+1}=A_nx_n+\sum_{j\in \mathbb{Z}} B_{n,j}x_j.
\end{equation*}
Of course, in such a setting the perturbed projectors $\widehat{\Pi}^\alpha$
does not a simple form but are "full" matrix operators.

To conclude this section, in view of Definition \ref{DE1.1}, we will
introduce without any proof the new class of sequence of closed linear
operators $\mathbf{A}=\left\{A_n\right\}_{n\in\mathbb{Z}}$ such that the
corresponding Howland evolution operator has an exponential trichotomy and
for which perturbation result (see Theorem \ref{TH1.8}) holds true. To do
let us consider for each $n\in\mathbb{Z}$ a closed linear operator $%
A_n:D(A_n)\subset Y\to Y$. We introduce the following definition:

\begin{definition}
\label{DEF-plus} We say that the sequence $\mathbf{A}=\left\{
A_{n}:D(A_n)\subset Y\to Y\right\} _{n\in \mathbb{Z}}$. $\mathbf{A}$ is 
\textbf{exponentially trichotomic}) on $\mathbb{Z}$ with constant $\kappa $,
exponents $0<\rho _{0}<\rho $ if there exist three families of projectors $%
\mathbf{\Pi }^{\alpha }=\left\{ \Pi _{n}^{\alpha }\right\} _{n\in \mathbb{Z}%
}:\mathbb{Z}\rightarrow \mathcal{L}(Y)$, with $\alpha =u,s,c$ satisfying the
following properties:

\begin{itemize}
\item[(i)] For all $n\in \mathbb{Z}$ and $\alpha ,\beta \in \left\{
u,s,c\right\} $ we have%
\begin{equation*}
\Pi _{n}^{\alpha }\Pi _{n}^{\beta }=0,\ \text{if }\alpha \neq \beta ,\text{
and } \Pi _{n}^{s}+\Pi _{n}^{u}+\Pi _{n}^{c}=I_Y.
\end{equation*}

\item[(ii)] For all $n$ $\in \mathbb{Z}$ we have $D(A_n)=\Pi_n^s(Y)\oplus
\Pi_n^c(Y)\oplus\left(D(A_n)\cap \Pi_n^u(Y)\right)$ and for each $%
\alpha=s,c,u$ $A_n\left(D(A_n)\cap \Pi_n^\alpha(Y)\right)\subset
\Pi_{n+1}^\alpha(Y)$.

\item[(iii)] For each $n\in\mathbb{Z}$ the operator $A_n^c:=\left(A_n%
\right)_{|\Pi_n^c(Y)}$ is invertible from $\Pi_n^c(Y)$ onto $\Pi_{n+1}^c(Y)$
and the operator $A_n^u:D\left(A_n^u\right)\subset \Pi_n^u(Y)\to
\Pi_{n+1}^u(Y)$ is invertible with $\left(A_n^u\right)^{-1}\in \mathcal{L}%
\left(\Pi_{n+1}^u(Y),\Pi_{n}^u(Y)\right)$.\newline
Now for each $n\geq m$ we set for $\alpha=s,c$: 
\begin{equation*}
U^\alpha_{\mathbf{A}}\left( n,m\right) :=
A^\alpha_{n-1}...A^\alpha_{m}\Pi^\alpha_m\text{, \ if }n>m,\text{ and }I_Y%
\text{ if $n=m$},
\end{equation*}%
and for each $n\geq m$ we set for $\alpha=u,c$: 
\begin{equation*}
U^\alpha_{\mathbf{A}}\left( m,n\right) :=
\left(A^\alpha_{m}\right)^{-1}...\left(A^\alpha_{n-1}\right)^{-1}\Pi_n^\alpha%
\text{, \ if }n>m,\text{ and }I_Y\text{ if $n=m$}.
\end{equation*}

\item[(iv)] For each $y\in Y$ we have for all $n,m\in \mathbb{Z}$%
\begin{equation*}
\left\Vert U_{\mathbf{A}}^{c}\left( n,m\right) \Pi _{m}^{c}y\right\Vert \leq
\kappa e^{\rho _{0}\left\vert n-m\right\vert }\left\Vert y\right\Vert ,
\end{equation*}%
and if $n\geq m$%
\begin{equation*}
\left\Vert U_{\mathbf{A}}^{s}\left( n,m\right) \Pi _{m}^{s}y\right\Vert \leq
\kappa e^{-\rho \left( n-m\right) }\left\Vert y\right\Vert ,
\end{equation*}%
\begin{equation*}
\left\Vert U_{\mathbf{A}}^{u}\left( m,n\right) \Pi _{n}^{u}y\right\Vert \leq
\kappa e^{-\rho \left( n-m\right) }\left\Vert y\right\Vert .
\end{equation*}
\end{itemize}
\end{definition}

Using the above definition, one can check that if $\mathbf{A}%
=\left\{A_n:D(A_n)\subset Y\to Y\right\}_{n\in\mathbb{Z}}$ is an
exponentially trichotomic sequence of closed linear operator then the linear
operator $\mathcal{A}:D\left(\mathcal{A}\right)\subset X\to X$ defined by 
\begin{equation*}
\begin{cases}
D\left(\mathcal{A}\right)=\left\{u\in X:\;u_k\in D(A_k),\;\forall k\in%
\mathbb{Z}\text{ and }\left(A_nu_n\right)_{n\in\mathbb{Z}}\in X\right\}, \\ 
\left(\mathcal{A}u\right)_k=A_{k-1}u_{k-1},\;k\in\mathbb{Z},\;u\in D\left(%
\mathcal{A}\right),%
\end{cases}%
\end{equation*}
has an exponential trichotomy according to Definition \ref{DE1.1}.

\section{Proof of Theorem \protect\ref{TH1.4}}

\subsection{A continuity projector lemma}

The following lemma is inspired from \cite[Lemma 4.1]{Bates-Lu-Zeng-1998}.

\begin{lemma}
\label{LE2.1}Let $\Pi :X\rightarrow X$ and $\widehat{\Pi }:X\rightarrow X$
be two bounded linear projectors on a Banach space $X$.\ Assume that 
\begin{equation}
\left\Vert \Pi -\widehat{\Pi }\right\Vert _{\mathcal{L}\left( X\right)
}<\delta \text{ with }0<\delta <\sqrt{2}-1,  \label{2.1}
\end{equation}%
Then $\Pi $ is invertible from $\widehat{\Pi }\left( X\right) $ into $\Pi
\left( X\right) $ and%
\begin{equation}
\left\Vert \left( \Pi |_{\widehat{\Pi }\left( X\right) }\right)
^{-1}x\right\Vert \leq \frac{1}{1-\delta }\left\Vert x\right\Vert ,\ \forall
x\in \Pi \left( X\right) .  \label{2.2}
\end{equation}
\end{lemma}

\begin{remark}
\label{REM2.2}By symmetry, the bounded linear projector $\widehat{\Pi }$ is
also invertible from $\Pi \left( X\right) $ into $\widehat{\Pi }\left(
X\right) $ and 
\begin{equation}
\left\Vert \left( \widehat{\Pi }|_{\Pi \left( X\right) }\right)
^{-1}x\right\Vert \leq \frac{1}{1-\delta }\left\Vert x\right\Vert ,\ \forall
x\in \widehat{\Pi }\left( X\right)  \label{2.3}
\end{equation}
\end{remark}

\begin{proof}
We will first prove two claims.

\begin{claim}
\label{CLAIM1} If $\widehat{\Pi }\Pi $ is invertible from $\widehat{\Pi }(X)$
into $\widehat{\Pi }\left( X\right) $ then $\widehat{\Pi }$ is onto from $%
\Pi \left( X\right) $ into $\widehat{\Pi }\left( X\right) .$
\end{claim}

\begin{proof}[Proof of Claim \protect\ref{CLAIM1}]
Let $y\in \widehat{\Pi }\left( X\right)$ be given. Since the map $\widehat{%
\Pi }\Pi $ is invertible on $\widehat{\Pi }\left( X\right) $, there exists a
unique $x\in \widehat{\Pi }\left( X\right) $ such that $\widehat{\Pi }\Pi
x=y $. Therefore by setting $\bar{x}=\Pi x\in \Pi \left( X\right) $ we have $%
\widehat{\Pi }\bar{x}=y$, which implies the surjectivity of $\widehat{\Pi }$
from $\Pi \left( X\right) $ onto $\widehat{\Pi }\left( X\right) .$
\end{proof}

\begin{claim}
\label{CLAIM2} If $\Pi \widehat{\Pi }$ is invertible from $\Pi (X)$ into $%
\Pi (X)$ then $\widehat{\Pi }$ is one to one from $\Pi (X)$ into $\widehat{%
\Pi }(X).$
\end{claim}

\begin{proof}[Proof of Claim \protect\ref{CLAIM2}]
Let $x\in\Pi(X)$ be given such that $\widehat{\Pi}x=0.$ Then we have $\Pi%
\widehat{\Pi}x=0$. Since $\Pi\widehat{\Pi}$ is invertible from $\Pi(X)$ into 
$\Pi(X)$ we deduce that $x=0$ and the claim follows.
\end{proof}

Let us now prove that $\widehat{\Pi}\Pi$ is invertible from $\widehat{\Pi}%
(X) $ into $\widehat{\Pi}(X)$.\ First note that one has 
\begin{equation*}
\widehat{\Pi}\Pi=I-\left( I-\widehat{\Pi}\Pi\right).
\end{equation*}
Hence it is sufficient to prove that 
\begin{equation}
\left\Vert I-\widehat{\Pi}\Pi\right\Vert _{\mathcal{L}\left( \widehat{\Pi }%
(X)\right) }<1.  \label{4}
\end{equation}
Let $x\in\widehat{\Pi}(X)$ be given. Then we have 
\begin{equation*}
x-\widehat{\Pi}\Pi x =\widehat{\Pi}x-\widehat{\Pi}\Pi x=\left[ \widehat{\Pi}%
-\Pi\right] x+\left[ \widehat{\Pi}-\Pi\right] \Pi x.
\end{equation*}
Thus 
\begin{equation*}
\left\Vert x-\widehat{\Pi}\Pi x\right\Vert \leq\left\Vert \widehat{\Pi}%
-\Pi\right\Vert _{\mathcal{L}\left( X\right) }\left\Vert x\right\Vert
+\left\Vert \widehat{\Pi}-\Pi\right\Vert _{\mathcal{L}\left( X\right)
}\left\Vert \Pi x\right\Vert,
\end{equation*}
and 
\begin{equation*}
\left\Vert x-\widehat{\Pi}\Pi x\right\Vert \leq\delta\left\Vert x\right\Vert
+\delta\left\Vert \Pi x\right\Vert .
\end{equation*}
Since $x\in\widehat{\Pi}(X)$ we have%
\begin{equation}  \label{aqzp}
\Pi x =\Pi x-x+x =\Pi x-\widehat{\Pi}x+x,
\end{equation}
hence 
\begin{equation*}
\left\Vert \Pi x\right\Vert \leq \left[\left\Vert \Pi-\widehat{\Pi}%
\right\Vert _{\mathcal{L}\left( X\right) }+1\right]\left\Vert
x\right\Vert\leq \left( 1+\delta\right) \left\Vert x\right\Vert.
\end{equation*}
Then we obtain 
\begin{equation*}
\left\Vert x-\widehat{\Pi}\Pi x\right\Vert \leq \delta\left( 2+\delta\right)
\left\Vert x\right\Vert,\;\;\forall x\in\widehat{\Pi}(X).
\end{equation*}
Recalling that $\delta\in\left( 0,\sqrt{2}-1\right)$, that reads $%
\delta\left( 2+\delta\right) <1$, and we deduce that $\widehat{\Pi}\Pi$ is
invertible from $\widehat{\Pi}(X)$ into $\widehat{\Pi}(X)$. By symmetry it
follows that $\Pi\widehat{\Pi }$ is also invertible from $\Pi(X)$ into $%
\Pi(X)$.

To conclude the proof let us estimate the norm of the inverse of $\Pi |_{%
\widehat{\Pi }\left( X\right) }.$ Let $x\in \widehat{\Pi }(X)$ be given.
From \eqref{aqzp} one has 
\begin{equation*}
\left\Vert \Pi x\right\Vert \geq \left\Vert x\right\Vert -\left\Vert 
\widehat{\Pi }x-\Pi x\right\Vert \geq \left\Vert x\right\Vert -\left\Vert 
\widehat{\Pi }-\Pi \right\Vert _{\mathcal{L}\left( X\right) }\left\Vert
x\right\Vert\geq \left( 1-\delta \right) \left\Vert x\right\Vert ,
\end{equation*}%
and the result follows.\ 
\end{proof}

\subsection{Derivation of the fixed point problem}

In this section we shall derive a fixed point formulation for perturbed
trichotomy. All the computations we will done for bounded linear operator.
However one could remark that the formulations summarized in the lemma below
makes sense for unbounded exponentially trichotomic linear operator as
defined in Definition \ref{DE1.1}. This will be used to prove Theorem \ref%
{TH1.4} for bounded perturbation of unbounded exponentially trichotomic
linear operator.

Recall the discrete time variation of constant formula for bounded linear
operators $A,B\in \mathcal{L}(X)$.\ We have 
\begin{eqnarray*}
\left( A+B\right) ^{n} &=&A\left( A+B\right) ^{n-1}+B\left( A+B\right) ^{n-1}
\\
&=&A^{2}(A+B)^{n-2}+AB\left( A+B\right) ^{n-2}+B\left( A+B\right) ^{n-1}
\end{eqnarray*}%
thus by induction 
\begin{equation}
\left( A+B\right) ^{n}=A^{n}+A^{n-1}B+...+AB\left( A+B\right) ^{n-2}+B\left(
A+B\right) ^{n-1},  \label{2.7}
\end{equation}%
so that for each $n\geq p$ we obtain 
\begin{equation}
\left( A+B\right) ^{n-p}=A^{n-p}+\sum_{m=p}^{n-1}A^{n-m-1}B\left( A+B\right)
^{m-p}.  \label{2.8}
\end{equation}

In the sequel and throughout this work we shall use the following summation
convention: 
\begin{equation*}
\sum_{n}^m=0\text{ if $m<n$}.
\end{equation*}
This notational convention is similar to the one used by Vanderbauwhede in 
\cite{Vanderbauwhede} who specified this using the symbol $\sum^{(+)}$.

Then using the above constant variation formula, one obtains the following
fixed point formulation for a perturbed trichotomic semiflow in the bounded
case:

\begin{lemma}
\label{LE2.2}Let $A\in \mathcal{L}(X)$ be given such that it has an
exponential trichotomy with constant $\kappa ,$ exponents $0<\rho _{0}<\rho $
and associated to the three projectors $\Pi _{k},$ $k=s,c,u$. Let $B\in 
\mathcal{L}(X)$ be given such that $A+B$ has an exponential trichotomy with
constant $\widehat{\kappa },$ exponents $0<\widehat{\rho }_{0}<\widehat{\rho 
}$ such that $\rho _{0}<\widehat{\rho }_{0}<\widehat{\rho }<\rho $ and
associated to the three projectors $\widehat{\Pi }_{k},$ $k=s,c,u$. Then one
has for each $n\in \mathbb{N},$%
\begin{eqnarray}
\left( A+B\right) _{s}^{n}\widehat{\Pi }_{s} &=&A_{s}^{n}\Pi _{s}\widehat{%
\Pi }_{s}+\sum_{m=0}^{n-1}A_{s}^{n-m-1}\Pi _{s}B\left( A+B\right) _{s}^{m}%
\widehat{\Pi }_{s}  \label{2.9} \\
&&-\sum_{m=0}^{+\infty }\left[ A_{u}^{-m-1}\Pi _{u}+A_{c}^{-m-1}\Pi _{c}%
\right] B\left( A+B\right) _{s}^{n+m}\widehat{\Pi }_{s},  \notag
\end{eqnarray}%
\begin{eqnarray}
\left( A+B\right) _{u}^{-n}\widehat{\Pi }_{u} &=&A_{u}^{-n}\Pi _{u}\widehat{%
\Pi }_{u}  \label{2.10} \\
&&-\sum_{m=0}^{n-1}\left[ A_{u}^{-m-1}\Pi _{u}\right] B\left( A+B\right)
_{u}^{m-n}\widehat{\Pi }_{u}  \notag \\
&&+\sum_{m=0}^{+\infty }\left[ A_{s}^{m}\Pi _{s}+A_{c}^{m}\Pi _{c}\right]
B\left( A+B\right) _{u}^{-m-1-n}\widehat{\Pi }_{u},  \notag
\end{eqnarray}%
\begin{eqnarray}
\left( A+B\right) _{c}^{n}\widehat{\Pi }_{c} &=&A_{c}^{n}\Pi _{c}\widehat{%
\Pi }_{c}+\sum_{m=0}^{n-1}A_{c}^{n-m-1}\Pi _{c}B\left( A+B\right) _{c}^{m}%
\widehat{\Pi }_{c}  \label{2.11} \\
&&-\sum_{m=0}^{+\infty }A_{u}^{-m-1}\Pi _{u}B\left( A+B\right) _{c}^{m+n}%
\widehat{\Pi }_{c}  \notag \\
&&+\sum_{m=0}^{+\infty }A_{s}^{m}\Pi _{s}B\left( A+B\right) _{c}^{-m-1+n}%
\widehat{\Pi }_{c}.  \notag
\end{eqnarray}%
\begin{eqnarray}
\left( A+B\right) _{c}^{-n}\widehat{\Pi }_{c} &=&A_{c}^{-n}\Pi _{c}\widehat{%
\Pi }_{c}-\sum_{m=0}^{n-1}A_{c}^{-m-1}\Pi _{c}B\left( A+B\right) _{c}^{m-n}%
\widehat{\Pi }_{c}  \label{2.12} \\
&&-\sum_{m=0}^{+\infty }A_{u}^{-m-1}\Pi _{u}B\left( A+B\right) _{c}^{m-n}%
\widehat{\Pi }_{c}  \notag \\
&&+\sum_{m=0}^{+\infty }A_{s}^{m}\Pi _{s}B\left( A+B\right) _{c}^{-m-1-n}%
\widehat{\Pi }_{c}.  \notag
\end{eqnarray}%
\begin{eqnarray}
\widehat{\Pi }_{s} &=&\Pi _{s}-\sum_{m=0}^{+\infty }A_{s}^{m}\Pi _{s}B\left[
\left( A+B\right) _{u}^{-m-1}\widehat{\Pi }_{u}+\left( A+B\right) _{c}^{-m-1}%
\widehat{\Pi }_{c}\right]  \label{2.13} \\
&&-\sum_{m=0}^{+\infty }\left[ A_{u}^{-m-1}\Pi _{u}+A_{c}^{-m-1}\Pi _{c}%
\right] B\left( A+B\right) _{s}^{m}\widehat{\Pi }_{s},  \notag
\end{eqnarray}%
\begin{eqnarray}
\widehat{\Pi }_{u} &=&\Pi _{u}+\sum_{m=0}^{+\infty }\left[ A_{c}^{m}\Pi
_{c}+A_{s}^{m}\Pi _{s}\right] B\left( A+B\right) _{u}^{-m-1}\widehat{\Pi }%
_{u}  \label{2.14} \\
&&+\sum_{m=0}^{+\infty }A_{u}^{-m-1}\Pi _{u}B\left[ \left( A+B\right)
_{s}^{m}\widehat{\Pi }_{s}+\left( A+B\right) _{c}^{m}\widehat{\Pi }_{c}%
\right] ,  \notag
\end{eqnarray}%
and%
\begin{eqnarray}
\widehat{\Pi }_{c} &=&I-\widehat{\Pi }_{s}-\widehat{\Pi }_{u}  \label{2.15}
\\
&=&\Pi _{c}-\sum_{m=0}^{+\infty }\left[ A_{c}^{m}\Pi _{c}+A_{s}^{m}\Pi _{s}%
\right] B\left( A+B\right) _{u}^{-m-1}\widehat{\Pi }_{s}  \notag \\
&&-\sum_{m=0}^{+\infty }A_{u}^{-m-1}\Pi _{u}B\left[ \left( A+B\right)
_{s}^{m}\widehat{\Pi }_{s}+\left( A+B\right) _{c}^{m}\widehat{\Pi }_{c}%
\right]  \notag \\
&&+\sum_{m=0}^{+\infty }A_{s}^{m}\Pi _{s}B\left[ \left( A+B\right)
_{u}^{-m-1}\widehat{\Pi }_{u}+\left( A+B\right) _{c}^{-m-1}\widehat{\Pi }_{c}%
\right]  \notag \\
&&+\sum_{m=0}^{+\infty }\left[ A_{u}^{-m-1}\Pi _{u}+A_{c}^{-m-1}\Pi _{c}%
\right] B\left( A+B\right) _{s}^{m}\widehat{\Pi }_{s}  \notag
\end{eqnarray}
\end{lemma}

\begin{proof}
\textbf{Derivation of formula (\ref{2.13}) for }$\widehat{\Pi }_{s}$\textbf{%
: }By applying $\widehat{\Pi }_{s}$ on the right side of (\ref{2.8}) we
obtain%
\begin{equation}
\left( A+B\right) _{s}^{n-p}\widehat{\Pi }_{s}=A^{n-p}\widehat{\Pi }%
_{s}+\sum_{m=p}^{n-1}A^{n-m-1}B\left( A+B\right) _{s}^{m-p}\widehat{\Pi }%
_{s},\ \forall n\geq p.  \label{2.16}
\end{equation}%
By fixing $p=0$ and by applying $A_{u}^{-n}\Pi _{u}$ on the left side of the
above formula we obtain%
\begin{equation}
A_{u}^{-n}\Pi _{u}\left( A+B\right) _{s}^{n}\widehat{\Pi }_{s}=\Pi _{u}%
\widehat{\Pi }_{s}+\sum_{m=0}^{n-1}A_{u}^{-m-1}\Pi _{u}B\left( A+B\right)
_{s}^{m}\widehat{\Pi }_{s}.  \label{2.17}
\end{equation}%
Since for each $n\geq 0$ one has 
\begin{equation*}
\left\Vert A_{u}^{-n}\Pi _{u}\right\Vert _{\mathcal{L}\left( X\right) }\leq
\kappa e^{-\rho n}\left\Vert \Pi _{u}\right\Vert _{\mathcal{L}\left(
X\right) }\text{ and } \left\Vert \left( A+B\right) _{s}^{n}\widehat{\Pi }%
_{s}\right\Vert _{\mathcal{L}\left( X\right) }\leq \kappa e^{-\widehat{\rho }%
n}\left\Vert \widehat{\Pi }_{s}\right\Vert _{\mathcal{L}\left( X\right) },
\end{equation*}%
by letting $n$ goes to $+\infty $ in (\ref{2.17}) it follows that%
\begin{equation}
\Pi _{u}\widehat{\Pi }_{s}=-\sum_{m=0}^{+\infty }A_{u}^{-m-1}\Pi _{u}B\left(
A+B\right) _{s}^{m}\widehat{\Pi }_{s}\text{.}  \label{2.18}
\end{equation}%
By fixing $p=0$ and by applying $A_{c}^{-n}\Pi _{c}$ (instead of $%
A_{u}^{-n}\Pi _{u}$) on the left side of (\ref{2.16}) we obtain 
\begin{equation}
\Pi _{c}\widehat{\Pi }_{s}=-\sum_{m=0}^{+\infty }A_{c}^{-m-1}\Pi _{c}B\left(
A+B\right) _{s}^{m}\widehat{\Pi }_{s}\text{.}  \label{2.19}
\end{equation}%
Then combining (\ref{2.18}) and (\ref{2.19}) leads us to 
\begin{equation}
\begin{array}{l}
\widehat{\Pi }_{s}=\Pi _{s}\widehat{\Pi }_{s}+\Pi _{u}\widehat{\Pi }_{s}+\Pi
_{c}\widehat{\Pi }_{s} \\ 
=\Pi _{s}\widehat{\Pi }_{s}-\overset{+\infty }{\underset{m=0}{\sum }}\left[
A_{u}^{-m-1}\Pi _{u}+A_{c}^{-m-1}\Pi _{c}\right] B\left( A+B\right) _{s}^{m}%
\widehat{\Pi }_{s}.%
\end{array}
\label{2.20}
\end{equation}%
It thus remains to reformulate $\Pi _{s}\widehat{\Pi }_{s}$ by using 
\begin{equation*}
\Pi _{s}\widehat{\Pi }_{s}=\Pi _{s}\left[ I-\widehat{\Pi }_{u}-\widehat{\Pi }%
_{c}\right] .
\end{equation*}%
Therefore we will compute $\Pi _{s}\widehat{\Pi }_{u}$ and $\Pi _{s}\widehat{%
\Pi }_{c}.$

\noindent \textbf{Computation of }$\Pi _{s}\widehat{\Pi }_{u}$\textbf{: }By
applying $\widehat{\Pi }_{u}$ on the right side of (\ref{2.8}) we have 
\begin{equation}
\left( A+B\right) _{u}^{n-p}\widehat{\Pi }_{u}=A^{n-p}\widehat{\Pi }%
_{u}+\sum_{m=p}^{n-1}A^{n-m-1}B\left( A+B\right) _{u}^{m-p}\widehat{\Pi }%
_{u},\ \forall n\geq p.  \label{2.21}
\end{equation}%
By applying $\Pi _{s}$ on the left side of the above formula we obtain 
\begin{equation}
\Pi _{s}\left( A+B\right) _{u}^{n-p}\widehat{\Pi }_{u}=A_{s}^{n-p}\Pi _{s}%
\widehat{\Pi }_{u}+\sum_{m=p}^{n-1}A_{s}^{n-m-1}\Pi _{s}B\left( A+B\right)
_{u}^{m-p}\widehat{\Pi }_{u},\ \forall n\geq p,  \label{2.22}
\end{equation}%
and by applying $\left( A+B\right) _{u}^{p-n}\widehat{\Pi }_{u}$ \ on the
right side of (\ref{2.22}) we have 
\begin{equation}
\Pi _{s}\widehat{\Pi }_{u}=A_{s}^{n-p}\Pi _{s}\left( A+B\right) _{u}^{p-n}%
\widehat{\Pi }_{u}+\sum_{m=p}^{n-1}A_{s}^{n-m-1}\Pi _{s}B\left( A+B\right)
_{u}^{m-n}\widehat{\Pi }_{u},\ \forall n\geq p,  \label{2.23}
\end{equation}%
and since 
\begin{equation*}
\left\Vert A_{s}^{n-p}\Pi _{s}\right\Vert _{\mathcal{L}\left( X\right) }\leq
\kappa e^{-\rho \left( n-p\right) }\left\Vert \Pi _{c}\right\Vert _{\mathcal{%
L}\left( X\right) }
\end{equation*}%
and 
\begin{equation*}
\left\Vert \left( A+B\right) _{u}^{p-n}\widehat{\Pi }_{u}\right\Vert _{%
\mathcal{L}\left( X\right) }\leq \kappa e^{-\widehat{\rho }\left( n-p\right)
}\left\Vert \widehat{\Pi }_{u}\right\Vert _{\mathcal{L}\left( X\right) },
\end{equation*}%
by taking the limit when $p$ goes to $-\infty $ in (\ref{2.23}) yields 
\begin{equation}
\Pi _{s}\widehat{\Pi }_{u}=\sum_{m=0}^{+\infty }A_{s}^{m}\Pi _{s}B\left(
A+B\right) _{u}^{-m-1}\widehat{\Pi }_{u}.  \label{2.24}
\end{equation}%
\noindent \textbf{Computation of }$\Pi _{s}\widehat{\Pi }_{c}$\textbf{: }%
Starting from the equality\textbf{\ }%
\begin{equation*}
\left( A+B\right) _{c}^{n-p}\widehat{\Pi }_{c}=A^{n-p}\widehat{\Pi }%
_{c}+\sum_{m=p}^{n-1}A^{n-m-1}B\left( A+B\right) _{c}^{m-p}\widehat{\Pi }%
_{c},\ \forall n\geq p.
\end{equation*}%
and applying $\left( A+B\right) _{c}^{p-n}\widehat{\Pi }_{u}$ on the right
side of this formula we obtain for each $n\geq p$%
\begin{equation}
\Pi _{s}\widehat{\Pi }_{c}=A_{s}^{n-p}\Pi _{s}\left( A+B\right) _{c}^{p-n}%
\widehat{\Pi }_{u}+\sum_{k=p}^{n-1}A_{s}^{n-m-1}\Pi _{s}B\left( A+B\right)
_{c}^{m-n}\widehat{\Pi }_{c},\ \forall n\geq p.  \label{2.25}
\end{equation}%
and since 
\begin{equation*}
\left\Vert A_{s}^{n-p}\Pi _{s}\right\Vert _{\mathcal{L}\left( X\right) }\leq
\kappa e^{-\rho \left( n-p\right) }\left\Vert \Pi _{c}\right\Vert _{\mathcal{%
L}\left( X\right) }\text{,}
\end{equation*}%
and%
\begin{equation*}
\left\Vert \left( A+B\right) _{u}^{p-n}\widehat{\Pi }_{u}\right\Vert _{%
\mathcal{L}\left( X\right) }\leq \kappa e^{\widehat{\rho }_{0}\left(
n-p\right) }\left\Vert \widehat{\Pi }_{u}\right\Vert _{\mathcal{L}\left(
X\right) },
\end{equation*}%
with $\widehat{\rho }_{0}<\rho ,$ by letting $p$ goes to $-\infty $ into (%
\ref{2.25}) we derive 
\begin{equation}
\Pi _{s}\widehat{\Pi }_{c}=\sum_{m=0}^{+\infty }A_{s}^{m}\Pi _{s}B\left(
A+B\right) _{c}^{-m-1}\widehat{\Pi }_{c}.  \label{2.26}
\end{equation}%
\noindent \textbf{Computation of }$\Pi _{s}\widehat{\Pi }_{s}$\textbf{: }By
summing (\ref{2.24}) and (\ref{2.26}) it follows that 
\begin{equation}
\Pi _{s}\left[ \widehat{\Pi }_{c}+\widehat{\Pi }_{u}\right]
=\sum_{m=0}^{+\infty }A_{s}^{m}\Pi _{s}B\left[ \left( A+B\right) _{c}^{-m-1}%
\widehat{\Pi }_{c}+\left( A+B\right) _{u}^{-m-1}\widehat{\Pi }_{u}\right] ,
\label{2.27}
\end{equation}%
and since $\widehat{\Pi }_{c}+\widehat{\Pi }_{u}=I-\widehat{\Pi }_{s}$ it
follows that 
\begin{equation}
\Pi _{s}\widehat{\Pi }_{s}=\Pi _{s}-\sum_{m=0}^{+\infty }A_{s}^{m}\Pi _{s}B%
\left[ \left( A+B\right) _{c}^{-m-1}\widehat{\Pi }_{c}+\left( A+B\right)
_{u}^{-m-1}\widehat{\Pi }_{u}\right] .  \label{2.28}
\end{equation}%
Finally the expression of $\widehat{\Pi }_{s}$ in (\ref{2.13}) follows by
combining (\ref{2.20}) and (\ref{2.28}).

\noindent \textbf{Computation of }$\widehat{\Pi }_{u}$\textbf{\ and }$%
\widehat{\Pi }_{c}$\textbf{: }The derivation of the formula (\ref{2.14}) for 
$\widehat{\Pi }_{u}$ uses the same arguments as for $\widehat{\Pi }_{s}$.\
The formula (\ref{2.15}) for $\widehat{\Pi }_{c}$ is obtained by using $%
\widehat{\Pi }_{c}=I-\widehat{\Pi }_{s}-\widehat{\Pi }_{u}.$

\noindent \textbf{Computation of }$\left( A+B\right) _{s}^{n}\widehat{\Pi }%
_{s}$\textbf{: }Next we derive (\ref{2.9}). By applying $\left( A+B\right)
_{s}^{n}\widehat{\Pi }_{s}$ on the right side of (\ref{2.13}) we obtain 
\begin{equation}
\left( A+B\right) _{s}^{n}\widehat{\Pi }_{s}=\Pi _{s}\left( A+B\right)
_{s}^{n}\widehat{\Pi }_{s}-\sum_{m=0}^{+\infty }\left[
A_{u}^{-m-1}+A_{c}^{-m-1}\right] B\left( A+B\right) _{s}^{m+n}\widehat{\Pi }%
_{s}.  \label{2.29}
\end{equation}%
In order to determine $\Pi _{s}\left( A+B\right) _{s}^{n}\widehat{\Pi }_{s}$%
, we apply $\Pi _{s}$ on the left side of (\ref{2.16}) and we obtain 
\begin{equation}
\Pi _{s}\left( A+B\right) _{s}^{n}\widehat{\Pi }_{s}=A_{s}^{n-p}\Pi _{s}%
\widehat{\Pi }_{s}+\sum_{m=0}^{n-1}A_{s}^{n-m-1}\Pi _{s}B\left( A+B\right)
_{s}^{m}\widehat{\Pi }_{s},\ \forall n\in \mathbb{N},  \label{2.30}
\end{equation}%
\newline
and (\ref{2.9}) follows.

\noindent \textbf{Computation of }$\left( A+B\right) _{c}^{n}\widehat{\Pi }%
_{c}$\textbf{\ for }$n\geq 0$\textbf{: }By applying $\left( A+B\right)
_{c}^{n}\widehat{\Pi }_{c}$ on the right side of (\ref{2.15}) we obtain for
each $n\in \mathbb{N}$%
\begin{eqnarray}
\left( A+B\right) _{c}^{n}\widehat{\Pi }_{c} &=&\Pi _{c}\left( A+B\right)
_{c}^{n}\widehat{\Pi }_{c}-\sum_{m=0}^{+\infty }A_{u}^{-m-1}\Pi _{u}B\left(
A+B\right) _{c}^{m+n}\widehat{\Pi }_{c}  \label{2.31} \\
&&+\sum_{m=0}^{+\infty }A_{s}^{m}\Pi _{s}B\left( A+B\right) _{c}^{-m-1+n}%
\widehat{\Pi }_{c}.  \notag
\end{eqnarray}%
Next we compute $\Pi _{c}\left( A+B\right) _{c}^{n}\widehat{\Pi }_{c}$.\ By
using the variation of constant formula (\ref{2.8}) with $p=0$, and applying 
$\widehat{\Pi }_{c}$ on the right side and $\Pi _{c}$ on the left side we
obtain 
\begin{equation}
\Pi _{c}\left( A+B\right) _{c}^{n}\widehat{\Pi }_{c}=A_{c}^{n}\Pi _{c}%
\widehat{\Pi }_{c}+\sum_{m=0}^{n-1}A_{c}^{n-m-1}\Pi _{c}B\left( A+B\right)
_{c}^{m}\widehat{\Pi }_{c},  \label{2.32}
\end{equation}%
and (\ref{2.11}) follows.

\noindent \textbf{Computation of }$\left( A+B\right) _{c}^{n}\widehat{\Pi }%
_{c}$ \textbf{for }$n\leq 0$\textbf{: }By applying $\left( A+B\right)
_{c}^{-n}\widehat{\Pi }_{c}$ on the right side of (\ref{2.15}) we obtain%
\begin{eqnarray}
\left( A+B\right) _{c}^{-n}\widehat{\Pi }_{c} &=&\Pi _{c}\left( A+B\right)
_{c}^{-n}\widehat{\Pi }_{c}  \label{2.33} \\
&&-\sum_{m=0}^{+\infty }A_{u}^{-m-1}\Pi _{u}B\left( A+B\right) _{c}^{m-n}%
\widehat{\Pi }_{c}  \notag \\
&&+\sum_{m=0}^{+\infty }A_{s}^{m}\Pi _{s}B\left( A+B\right) _{c}^{-m-1-n}%
\widehat{\Pi }_{c}.  \notag
\end{eqnarray}%
Next we compute $\Pi _{c}\left( A+B\right) _{c}^{-n}\widehat{\Pi }_{c}$.\ By
applying $\left( A+B\right) _{c}^{-n}\widehat{\Pi }_{c}$ on the right side
of the variation of constant formula (\ref{2.8}) (with $p=0$) we obtain%
\begin{equation}
\widehat{\Pi }_{c}=A^{n}\left( A+B\right) _{c}^{-n}\widehat{\Pi }%
_{c}+\sum_{m=0}^{n-1}A^{n-m-1}B\left( A+B\right) _{c}^{m-n}\widehat{\Pi }%
_{c}.  \label{2.34}
\end{equation}%
By applying $A_{c}^{-n}\Pi _{c}$ on the left side of the above formula we
get 
\begin{equation}
A_{c}^{-n}\Pi _{c}\widehat{\Pi }_{c}=\Pi _{c}\left( A+B\right) _{c}^{-n}%
\widehat{\Pi }_{c}+\sum_{m=0}^{n-1}A_{c}^{-m-1}\Pi _{c}B\left( A+B\right)
_{c}^{m-n}\widehat{\Pi }_{c},  \label{2.35}
\end{equation}%
and the result follows.
\end{proof}

\subsection{Abstract reformulation of the fixed point problem}

In this section we reformulate the fixed point problem (\ref{2.9})-(\ref%
{2.14}) by using an abstract fixed point formulation.

Let $\eta >0$ be given. Define 
\begin{equation*}
\mathbb{L}_{-\eta }\left( \mathbb{N},\mathcal{L}\left( X\right) \right)
:=\left\{ \mathbf{u}:\mathbb{N}\rightarrow \mathcal{L}\left( X\right)
:\sup_{n\in \mathbb{N}}e^{\eta n}\left\Vert u_{n}\right\Vert <+\infty
\right\} ,
\end{equation*}%
which is a Banach space endowed with the norm 
\begin{equation*}
\left\Vert \mathbf{u}\right\Vert _{\mathbb{L}_{-\eta }}:=\sup_{n\in \mathbb{N%
}}e^{\eta n}\left\Vert u_{n}\right\Vert .
\end{equation*}%
Define 
\begin{equation*}
\mathbb{L}_{\eta }\left( \mathbb{Z},\mathcal{L}\left( X\right) \right)
:=\left\{ \mathbf{v}:\mathbb{Z}\rightarrow \mathcal{L}\left( X\right)
:\sup_{n\in \mathbb{Z}}e^{-\eta \left\vert n\right\vert }\left\Vert
v_{n}\right\Vert _{\mathcal{L}\left( X\right) }<+\infty \right\}
\end{equation*}%
which is a Banach space endowed with the norm 
\begin{equation*}
\left\Vert \mathbf{v}\right\Vert _{\mathbb{L}_{\eta }}:=\sup_{n\in \mathbb{Z}%
}e^{-\eta \left\vert n\right\vert }\left\Vert v_{n}\right\Vert _{\mathcal{L}%
\left( X\right) }.
\end{equation*}%
Consider $S_{-}$ the shift operators on $\mathbb{L}_{\pm \eta }\left( 
\mathbb{N},\mathcal{L}\left( X\right) \right) $ 
\begin{equation*}
S_{-}\left( \mathbf{u}\right) _{n}=u_{n+1}\text{ whenever }n\in \mathbb{Z}%
\text{ or }n\in \mathbb{N}\text{.}
\end{equation*}%
Let $C\in \mathcal{L}\left( X\right) $. In the following we will use the
linear operators 
\begin{equation*}
\Phi _{C}(\mathbf{u})_{n}=\sum_{m=0}^{n-1}C^{m}u_{n-1-m},\text{ and }\Theta
_{C}(\mathbf{u})_{n}=\sum_{m=0}^{+\infty }C^{m}u_{n+m}.
\end{equation*}%
\noindent \textbf{Reformulation of equation (\ref{2.9}) on }$\widehat{X}_{s}$%
\textbf{:} Set for each $n\in\mathbb{N}$: $E_{n}^{s}:=\left( A+B\right)
_{s}^{n}\widehat{\Pi }_{s}$. We require $E^{s}\in \mathbb{L}_{-\widehat{\rho 
}}\left( \mathbb{N},\mathcal{L}\left( X\right) \right)$, where $\widehat{%
\rho }$ is the constant introduced in Theorem \ref{TH1.4}. Consider the
linear operators $\Phi _{s},\;\Theta _{cu}:\mathbb{L}_{-\widehat{\rho }%
}\left( \mathbb{N},\mathcal{L}\left( X\right) \right) \rightarrow \mathbb{L}%
_{-\widehat{\rho }}\left( \mathbb{N},\mathcal{L}\left( X\right) \right) $
defined by 
\begin{equation*}
\Phi _{s}=\Phi _{A_{s}\Pi _{s}} \text{ and } \Theta _{cu}=\Theta _{\left(
A_{c}^{-1}\Pi _{c}+A_{u}^{-1}\Pi _{u}\right) }.
\end{equation*}%
We observe that 
\begin{equation*}
\Phi _{s}\circ B\circ \left( E^{s}\right) _{n}=\sum_{l=0}^{n-1}A_{s}^{l}\Pi
_{s}BE_{n}^{s}\left( A+B\right) _{s}^{n-1-l}\widehat{\Pi }%
_{s}=\sum_{m=0}^{n-1}A_{s}^{n-m-1}\Pi _{s}B\left( A+B\right) _{s}^{m}%
\widehat{\Pi }_{s}
\end{equation*}%
therefore the equation (\ref{2.9}) can be rewritten for $n\in \mathbb{N}$ as 
\begin{equation}
E_{n}^{s}=A_{s}^{n}\Pi _{s}\widehat{\Pi }_{s}+\Phi _{s}\circ B\circ \left(
E^{s}\right) _{n}-\Theta _{cu}(\left( A_{u}^{-1}\Pi _{u}+A_{c}^{-1}\Pi
_{c}\right) BE^{s})_{n}.  \label{2.36}
\end{equation}%
In order to solve the fixed point problem we will use the following lemma.\ 

\begin{lemma}
\label{LE2.3}The operators $\Phi _{s}$ and $\Theta _{cu}$ map $\mathbb{L}_{-%
\widehat{\rho }}\left( \mathbb{N},\mathcal{L}\left( X\right) \right) $ into
itself and are bounded linear operators on $\mathbb{L}_{-\widehat{\rho }%
}\left( \mathbb{N},\mathcal{L}\left( X\right) \right) $. More precisely we
have%
\begin{equation*}
\left\Vert \Phi _{s}\left( \mathbf{u}\right) \right\Vert _{\mathbb{L}_{-%
\widehat{\rho }}}\leq \frac{\kappa e^{\widehat{\rho }}}{1-e^{\widehat{\rho }%
-\rho }}\left\Vert \mathbf{u}\right\Vert _{\mathbb{L}_{-\widehat{\rho }}},\
\forall \mathbf{u}\in \mathbb{L}_{-\widehat{\rho }}\left( \mathbb{N},%
\mathcal{L}\left( X\right) \right) ,
\end{equation*}%
and%
\begin{equation*}
\left\Vert \Theta _{cu}\left( \mathbf{u}\right) \right\Vert _{\mathbb{L}_{-%
\widehat{\rho }}}\leq \left[ \frac{\kappa }{1-e^{\rho _{0}-\widehat{\rho }}}+%
\frac{\kappa }{1-e^{-\left( \rho +\widehat{\rho }\right) }}\right]
\left\Vert \mathbf{u}\right\Vert _{\mathbb{L}_{-\widehat{\rho }}},\ \forall 
\mathbf{u}\in \mathbb{L}_{-\widehat{\rho }}\left( \mathbb{N},\mathcal{L}%
\left( X\right) \right) .
\end{equation*}
\end{lemma}

\bigskip

\noindent \textbf{Reformulation of equation (\ref{2.10}) on }$\widehat{X}%
_{u} $\textbf{:} Set for each $n\in \mathbb{N}$: $E_{n}^{u}:=\left(
A+B\right) _{u}^{-n}\widehat{\Pi }_{u}$. We require $E^{u}\in \mathbb{L}_{-%
\widehat{\rho }}\left( \mathbb{N},\mathcal{L}\left( X\right) \right)$, where 
$\widehat{\rho }$ is the constant introduced in Theorem \ref{TH1.4}.
Consider the linear operators $\Phi _{u},\;\Theta _{sc}:\mathbb{L}_{-%
\widehat{\rho }}\left( \mathbb{N},\mathcal{L}\left( X\right) \right)
\rightarrow \mathbb{L}_{-\widehat{\rho }}\left( \mathbb{N},\mathcal{L}\left(
X\right) \right) $ 
\begin{equation*}
\Phi _{u}:=\Phi _{A_{u}^{-1}\Pi _{u}} \text{ and } \Theta _{sc}:=\Theta
_{\left( A_{s}\Pi _{s}+A_{c}\Pi _{c}\right) }.
\end{equation*}%
We observe that 
\begin{eqnarray*}
\Phi _{u}\circ \left( A_{u}^{-1}\Pi _{u}B\right) \circ S_{-}\left(
E^{u}\right) _{n} &=&\sum_{m=0}^{n-1}A_{u}^{-m}\Pi _{u}\left( A_{u}^{-1}\Pi
_{u}B\right) E_{n-m}^{u} \\
&=&\sum_{m=0}^{n-1}\left[ A_{u}^{-m-1}\Pi _{u}\right] B\left( A+B\right)
_{u}^{m-n}\widehat{\Pi }_{u}
\end{eqnarray*}%
therefore equation (\ref{2.10}) can be rewritten for each $n\in \mathbb{N}$
as 
\begin{equation}
E_{n}^{u}=A_{u}^{-n}\Pi _{u}\widehat{\Pi }_{u}-\Phi _{u}\circ \left(
A_{u}^{-1}\Pi _{u}B\right) \circ S_{-}\left( E^{u}\right) _{n}+\Theta
_{sc}\circ B\circ S_{-}\left( E^{u}\right) _{n}.  \label{2.37}
\end{equation}

\begin{lemma}
\label{LE2.4}The operators $\Phi _{u}$ and $\Theta _{sc}$ map $\mathbb{L}_{-%
\widehat{\rho }}\left( \mathbb{N},\mathcal{L}\left( X\right) \right) $ into
itself and are bounded linear operators on $\mathbb{L}_{-\widehat{\rho }%
}\left( \mathbb{N},\mathcal{L}\left( X\right) \right) $. More precisely we
have%
\begin{equation*}
\left\Vert \Phi _{u}\left( \mathbf{u}\right) \right\Vert _{\mathbb{L}_{-%
\widehat{\rho }}}\leq \frac{\kappa e^{\widehat{\rho }}}{1-e^{\widehat{\rho }%
-\rho }}\left\Vert \mathbf{u}\right\Vert _{\mathbb{L}_{-\widehat{\rho }}},\
\forall \mathbf{u}\in \mathbb{L}_{-\widehat{\rho }}\left( \mathbb{N},%
\mathcal{L}\left( X\right) \right) ,
\end{equation*}%
and%
\begin{equation*}
\left\Vert \Theta _{sc}\left( \mathbf{u}\right) \right\Vert _{\mathbb{L}_{-%
\widehat{\rho }}}\leq \left[ \frac{\kappa }{1-e^{-\left( \rho +\widehat{\rho 
}\right) }}+\frac{\kappa }{1-e^{\rho _{0}-\widehat{\rho }}}\right]
\left\Vert \mathbf{u}\right\Vert _{\mathbb{L}_{-\widehat{\rho }}},\ \forall 
\mathbf{u}\in \mathbb{L}_{-\widehat{\rho }}\left( \mathbb{N},\mathcal{L}%
\left( X\right) \right) .
\end{equation*}
\end{lemma}

\bigskip

\noindent \textbf{Reformulation of equation (\ref{2.11})-(\ref{2.12}) on }$%
\widehat{X}_{c}$\textbf{: }Set for each $n\in\mathbb{Z}$: $E_{n}^{c}:=\left(
A+B\right) _{c}^{n}\widehat{\Pi }_{c}$. We require $E^{c}\in \mathbb{L}_{%
\widehat{\rho }_{0}}\left( \mathbb{Z},\mathcal{L}\left( X\right) \right)$.
Define the linear operators 
\begin{equation*}
\Phi _{c}(\mathbf{u})_{n}:=\left\{ 
\begin{array}{l}
\overset{n-1}{\underset{m=0}{\sum }}A_{c}^{n-m-1}\Pi _{c}u_{m},\text{ if }%
n\geq 0 \\ 
-\overset{-n-1}{\underset{m=0}{\sum }}A_{c}^{-m-1}\Pi _{c}u_{m+n},\text{ if }%
n\leq 0%
\end{array}%
\right.
\end{equation*}%
and%
\begin{equation*}
\Theta _{su}(\mathbf{u})_{n}:=-\sum_{m=0}^{+\infty }A_{u}^{-m-1}\Pi
_{u}u_{m+n}+\sum_{m=0}^{+\infty }A_{s}^{m}\Pi _{s}u_{n-1-m},\text{ for }n\in 
\mathbb{Z},
\end{equation*}%
therefore equations (\ref{2.11})-(\ref{2.12}\textbf{)} can be rewritten for
each $n\in \mathbb{Z}$ as%
\begin{equation}
E_{n}^{c}:=A_{c}^{n}\Pi _{c}\left[ I-\left( \widehat{\Pi }_{s}+\widehat{\Pi }%
_{u}\right) \right] +\Phi _{c}(BE^{c})_{n}+\Theta _{su}(BE^{c})_{n}.
\label{2.38}
\end{equation}

\begin{lemma}
\label{LE2.5}The operators $\Phi _{c}$ and $\Theta _{su}$ map $\mathbb{L}_{%
\widehat{\rho }_{0}}\left( \mathbb{Z},\mathcal{L}\left( X\right) \right) $
into itself and are bounded linear operators on $\mathbb{L}_{\widehat{\rho }%
_{0}}\left( \mathbb{Z},\mathcal{L}\left( X\right) \right) $. More precisely
we have%
\begin{equation*}
\left\Vert \Phi _{c}\left( \mathbf{v}\right) \right\Vert _{\mathbb{L}_{%
\widehat{\rho }_{0}}}\leq \frac{\kappa }{1-e^{\rho _{0}-\widehat{\rho }_{0}}}%
\left\Vert \mathbf{v}\right\Vert _{\mathbb{L}_{\widehat{\rho }_{0}}},\
\forall \mathbf{v}\in \mathbb{L}_{\widehat{\rho }_{0}}\left( \mathbb{Z},%
\mathcal{L}\left( X\right) \right) ,
\end{equation*}%
and%
\begin{equation*}
\left\Vert \Theta _{su}\left( \mathbf{v}\right) \right\Vert _{\mathbb{L}_{%
\widehat{\rho }_{0}}}\leq \left[ \frac{\kappa e^{\widehat{\rho }}}{1-e^{%
\widehat{\rho }_{0}-\widehat{\rho }}}+\frac{\kappa e^{-\widehat{\rho }}}{%
1-e^{\widehat{\rho }_{0}-\widehat{\rho }}}\right] \left\Vert \mathbf{v}%
\right\Vert _{\mathbb{L}_{\widehat{\rho }_{0}}},\ \forall \mathbf{v}\in 
\mathbb{L}_{\widehat{\rho }_{0}}\left( \mathbb{Z},\mathcal{L}\left( X\right)
\right) .
\end{equation*}
\end{lemma}

\bigskip

\noindent \textbf{Reformulation of equation (\ref{2.13})-(\ref{2.14}) for
the projectors on }$\widehat{X}_{s}$\textbf{\ and }$\widehat{X}_{u}$\textbf{:%
} Define the linear operator \textbf{\ }%
\begin{equation*}
\Theta _{s}(u):=\Theta _{A_{s}\Pi _{s}}(u)_{0}=\sum_{m=0}^{+\infty
}A_{s}^{m}\Pi _{s}u_{m}
\end{equation*}%
then equation (\ref{2.13}) becomes 
\begin{equation}
\widehat{\Pi }_{s}=\Pi _{s}-\Theta _{s}\circ B\circ S_{-}(E^{u}+\chi
_{-}\left( E^{c}\right) )+\Theta _{cu}(\left( A_{u}^{-1}\Pi
_{u}+A_{c}^{-1}\Pi _{c}\right) BE^{s})_{0}  \label{2.39}
\end{equation}%
where $\chi _{-}:\mathbb{L}_{\widehat{\rho }_{0}}\left( \mathbb{Z},\mathcal{L%
}\left( X\right) \right) \rightarrow \mathbb{L}_{\widehat{\rho }_{0}}\left( 
\mathbb{N},\mathcal{L}\left( X\right) \right) $ is defined by%
\begin{equation*}
\chi _{-}\left( E^{c}\right) _{n}=E_{-n}^{c}\text{ for }n\geq 0.
\end{equation*}%
Define 
\begin{equation*}
\Theta _{u}(u):=\Theta _{A_{u}^{-1}\Pi _{u}}(u)_{0}=\sum_{m=0}^{+\infty
}A_{u}^{-m}\Pi _{u}u_{m},
\end{equation*}%
and \eqref{2.14} re-writes as: 
\begin{equation}
\widehat{\Pi }_{u}=\Pi _{u}+\Theta _{sc}\circ B\circ S_{-}\left(
E^{u}\right) _{0}+\Theta _{u}(A_{u}^{-1}\Pi _{u}B\left( E^{s}+\chi
_{+}\left( E^{c}\right) \right) ),  \label{2.40}
\end{equation}%
where $\chi _{+}:\mathbb{L}_{\widehat{\rho }_{0}}\left( \mathbb{Z},\mathcal{L%
}\left( X\right) \right) \rightarrow \mathbb{L}_{\widehat{\rho }_{0}}\left( 
\mathbb{N},\mathcal{L}\left( X\right) \right) $ is defined by%
\begin{equation*}
\chi _{+}\left( E^{c}\right) _{n}:=E_{n}^{c}\text{ for }n\geq 0.
\end{equation*}

\begin{lemma}
\label{LE2.6}The operators $\Theta _{s}$ and $\Theta _{u}$ have the
following properties:

\begin{itemize}
\item[(i)] $\Theta _{s}$ and $\Theta _{u}$ map $\mathbb{L}_{\widehat{\rho }%
_{0}}\left( \mathbb{N},\mathcal{L}\left( X\right) \right) $ into $\mathcal{L}%
\left( X\right) $ with 
\begin{equation*}
\left\Vert \Theta _{s}\left( \mathbf{v}\right) \right\Vert _{\mathcal{L}%
\left( X\right) }\leq \frac{\kappa }{1-e^{\widehat{\rho }_{0}-\rho }}%
\left\Vert \mathbf{v}\right\Vert _{\mathbb{L}_{\widehat{\rho }_{0}}},\
\forall \mathbf{v}\in \mathbb{L}_{\widehat{\rho }_{0}}\left( \mathbb{N},%
\mathcal{L}\left( X\right) \right) ,
\end{equation*}%
and%
\begin{equation*}
\left\Vert \Theta _{u}\left( \mathbf{v}\right) \right\Vert _{\mathcal{L}%
\left( X\right) }\leq \frac{\kappa }{1-e^{\widehat{\rho }_{0}-\rho }}%
\left\Vert \mathbf{v}\right\Vert _{\mathbb{L}_{\widehat{\rho }_{0}}},\
\forall \mathbf{v}\in \mathbb{L}_{\widehat{\rho }_{0}}\left( \mathbb{N},%
\mathcal{L}\left( X\right) \right) ;
\end{equation*}

\item[(ii)] $\Theta _{s}$ and $\Theta _{u}$ map $\mathbb{L}_{-\widehat{\rho }%
}\left( \mathbb{N},\mathcal{L}\left( X\right) \right) \subset \mathbb{L}_{%
\widehat{\rho }_{0}}\left( \mathbb{N},\mathcal{L}\left( X\right) \right) $
into $\mathcal{L}\left( X\right) $ with%
\begin{equation*}
\left\Vert \Theta _{s}\left( \mathbf{v}\right) \right\Vert _{\mathcal{L}%
\left( X\right) }\leq \frac{\kappa }{1-e^{-\widehat{\rho }-\rho }}\left\Vert 
\mathbf{v}\right\Vert _{\mathbb{L}_{-\widehat{\rho }}},\ \forall \mathbf{v}%
\in \mathbb{L}_{-\widehat{\rho }}\left( \mathbb{N},\mathcal{L}\left(
X\right) \right) ,
\end{equation*}%
\begin{equation*}
\left\Vert \Theta _{u}\left( \mathbf{v}\right) \right\Vert _{\mathcal{L}%
\left( X\right) }\leq \frac{\kappa }{1-e^{-\widehat{\rho }-\rho }}\left\Vert 
\mathbf{v}\right\Vert _{\mathbb{L}_{-\widehat{\rho }}},\ \forall \mathbf{v}%
\in \mathbb{L}_{-\widehat{\rho }}\left( \mathbb{N},\mathcal{L}\left(
X\right) \right) .
\end{equation*}
\end{itemize}
\end{lemma}

\bigskip

By using the expressions of $\widehat{\Pi }_{s}$ and $\widehat{\Pi }_{u}$
obtained in (\ref{2.39}) and (\ref{2.40}), and by replacing those
expressions into $A_{s}^{n}\Pi _{s}\widehat{\Pi }_{s}$ (respectively into $%
A_{u}^{-n}\Pi _{u}\widehat{\Pi }_{u}$ and $A_{c}^{n}\Pi _{c}\left[ I-\left( 
\widehat{\Pi }_{s}+\widehat{\Pi }_{u}\right) \right] $) in equation (\ref%
{2.36}) (respectively in (\ref{2.37}) and (\ref{2.38})) we will derive a new
fixed point problem only for $E^{s},\ E^{u}$ and $E^{c}$ given explicitly as
follow for each $n\in\mathbb{N}$: 
\begin{eqnarray}
E_{n}^{s} &=&A_{s}^{n}\Pi _{s}-\sum_{m=0}^{+\infty }A_{s}^{m+n}\Pi _{s}B 
\left[ E_{m+1}^{u}+E_{-m-1}^{c}\right]  \label{2.41} \\
&&+\sum_{m=0}^{n-1}A_{s}^{n-m-1}\Pi _{s}BE_{m}^{s}-\sum_{m=0}^{+\infty } 
\left[ A_{u}^{-m-1}\Pi _{u}+A_{c}^{-m-1}\Pi _{c}\right] BE_{n+m}^{s},  \notag
\end{eqnarray}%
\begin{eqnarray}
E_{n}^{u} &=&A_{u}^{-n}\Pi _{u}+\sum_{m=0}^{+\infty }A_{u}^{-n-m-1}\Pi _{u}B 
\left[ E_{m}^{s}+E_{m}^{c}\right]  \label{2.42} \\
&&-\sum_{m=0}^{n-1}\left[ A_{u}^{-m-1}\Pi _{u}\right] BE_{n-m}^{u}+%
\sum_{m=0}^{+\infty }\left[ A_{s}^{m}\Pi _{s}+A_{c}^{m}\Pi _{c}\right]
BE_{n+m+1}^{u},  \notag
\end{eqnarray}%
and for each $n\in \mathbb{Z}$: 
\begin{eqnarray}
E_{n}^{c} &=&A_{c}^{n}\Pi _{c}-\sum_{m=0}^{+\infty }A_{c}^{m+n}\Pi
_{c}BE_{m+1}^{u}+\sum_{m=0}^{+\infty }A_{c}^{n-m-1}\Pi _{c}BE_{m}^{s}
\label{2.43} \\
&&+\sum_{m=0}^{n-1}A_{c}^{n-m-1}\Pi
_{c}BE_{m}^{c}-\sum_{m=0}^{-n-1}A_{c}^{-m-1}\Pi _{c}BE_{n+m}^{c}  \notag \\
&&-\sum_{m=0}^{+\infty }A_{u}^{-m-1}\Pi _{u}BE_{n+m}^{c}+\sum_{m=0}^{+\infty
}A_{s}^{m}\Pi _{s}BE_{-m-1+n}^{c}.  \notag
\end{eqnarray}%
Moreover with this notation the explicit formulas for $\widehat{\Pi }_{k},\
k=s,c,u$ reads as 
\begin{equation*}
\begin{split}
&\widehat{\Pi }_{s}=\Pi _{s}-\sum_{m=0}^{+\infty }A_{s}^{m}\Pi _{s}B\left[
E_{m+1}^{u}+E_{-m-1}^{c}\right]-\sum_{m=0}^{+\infty }\left[ A_{u}^{-m-1}\Pi
_{u}+A_{c}^{-m-1}\Pi _{c}\right] BE_{m}^{s}, \\
&\widehat{\Pi }_{u}=\Pi _{u}+\sum_{m=0}^{+\infty }A_{u}^{-m-1}\Pi _{u}B\left[
E_{m}^{s}+E_{m}^{c}\right]+\sum_{m=0}^{+\infty }\left[ A_{s}^{m}\Pi
_{s}+A_{c}^{m}\Pi _{c}\right] E_{m+1}^{u}, \\
&\widehat{\Pi }_{c} =\Pi _{c}-\sum_{m=0}^{+\infty }A_{c}^{m}\Pi
_{c}BE_{m+1}^{u}+\sum_{m=0}^{+\infty }A_{c}^{-m-1}\Pi _{c}BE_{m}^{s} \\
&\phantom{\widehat{\Pi }_{c} =}-\sum_{m=0}^{+\infty }A_{u}^{-m-1}\Pi
_{u}BE_{m}^{c}+\sum_{m=0}^{+\infty }A_{s}^{m}\Pi _{s}BE_{-m-1}^{c}.
\end{split}%
\end{equation*}
Observe that we have the following relation $E_{0}^{k}=\widehat{\Pi }_{k}$
for each $k=s,c,u$, as well as the following identity 
\begin{equation*}
E_{0}^{s}+E_{0}^{c}+E_{0}^{u}=\widehat{\Pi }_{s}+\widehat{\Pi }_{c}+\widehat{%
\Pi }_{u}=I.
\end{equation*}%
Furthermore the system (\ref{2.41})-(\ref{2.43}) also re-writes as the
following compact form 
\begin{equation}
\left( 
\begin{array}{c}
E^{s} \\ 
E^{u} \\ 
E^{c}%
\end{array}%
\right) =\left( 
\begin{array}{c}
A_{s}^{\cdot }\Pi _{s} \\ 
A_{u}^{-\cdot }\Pi _{u} \\ 
A_{c}^{\cdot }\Pi _{c}%
\end{array}%
\right) +\mathcal{J }\left( 
\begin{array}{c}
E^{s} \\ 
E^{u} \\ 
E^{c}%
\end{array}%
\right),\text{ and }\mathcal{J}:=\left(\mathcal{J}_{ij}\right)_{1\leq
i,j\leq 3}  \label{2.49}
\end{equation}%
wherein the linear operators $\left\{ \mathcal{J}_{ij}\right\} _{1\leq
i,j\leq 3}$ are given by%
\begin{equation*}
\begin{split}
\mathcal{J}_{11}:=& A_{s}^{\cdot }\Pi _{s}\circ \Theta _{cu}\left( \cdot
\right) _{0}\circ \left( A_{u}^{-1}\Pi _{u}+A_{c}^{-1}\Pi _{c}\right) \circ
B+\Phi _{s}\circ B \\
& -\Theta _{cu}\circ \left( A_{u}^{-1}\Pi _{u}+A_{c}^{-1}\Pi _{c}\right)
\circ B \\
\mathcal{J}_{12}:=& -A_{s}^{\cdot }\Pi _{s}\circ \Theta _{s}\circ B\circ
S_{-} \\
\mathcal{J}_{13}:=& -A_{s}^{\cdot }\Pi _{s}\circ \Theta _{s}\circ B\circ
S_{-}\circ \chi _{-}
\end{split}%
\end{equation*}%
\begin{equation*}
\begin{split}
\mathcal{J}_{21}:=& A_{u}^{-\cdot }\Pi _{u}\circ \Theta _{u}\circ
A_{u}^{-1}\Pi _{u}\circ B \\
\mathcal{J}_{22}:=& A_{u}^{-\cdot }\Pi _{u}\circ \Theta _{sc}\left( \cdot
\right) _{0}\circ B\circ S_{-}-\Phi _{u}\circ A_{u}^{-1}\Pi _{u}\circ B\circ
S_{-} \\
& +\Theta _{sc}\circ B\circ S_{-} \\
\mathcal{J}_{23}:=& A_{u}^{-\cdot }\Pi _{u}\circ \Theta _{u}\circ
A_{u}^{-1}\Pi _{u}\circ B\circ \chi _{+} \\
&
\end{split}%
\end{equation*}%
\begin{equation*}
\begin{split}
\mathcal{J}_{31}:=& -A_{c}^{\cdot }\Pi _{c}\circ \Theta _{cu}(\cdot
)_{0}\circ \left( A_{u}^{-1}\Pi _{u}+A_{c}^{-1}\Pi _{c}\right) \circ B \\
& -A_{c}^{\cdot }\Pi _{c}\circ \Theta _{u}(\cdot )\circ A_{u}^{-1}\Pi
_{u}\circ B \\
\mathcal{J}_{32}:=& A_{c}^{\cdot }\Pi _{c}\circ \Theta _{s}\circ B\circ
S_{-}-A_{c}^{\cdot }\Pi _{c}\circ \Theta _{sc}\left( \cdot \right) _{0}\circ
B\circ S_{-} \\
\mathcal{J}_{33}:=& A_{c}^{\cdot }\Pi _{c}\circ \Theta _{s}\circ B\circ
S_{-}\circ \chi _{-} \\
& A_{c}^{\cdot }\Pi _{c}\circ \Theta _{u}\circ A_{u}^{-1}\Pi _{u}\circ
B\circ \chi _{+}+\Phi _{c}\circ B+\Theta _{su}\circ B.
\end{split}%
\end{equation*}%
In the sequel we define the Banach space 
\begin{equation*}
\mathcal{X}=\mathbb{L}_{-\widehat{\rho }}\left( \mathbb{N},\mathcal{L}\left(
X\right) \right) \times \mathbb{L}_{-\widehat{\rho }}\left( \mathbb{N},%
\mathcal{L}\left( X\right) \right) \times \mathbb{L}_{\widehat{\rho }%
_{0}}\left( \mathbb{Z},\mathcal{L}\left( X\right) \right),
\end{equation*}
endowed with the usual product norm: 
\begin{equation*}
\left\Vert \mathbf{Z}\right\Vert_{\mathcal{X}} =\max \left\{ \left\Vert
E^{s}\right\Vert _{\mathbb{L}_{-\widehat{\rho }}},\left\Vert
E^{u}\right\Vert _{\mathbb{L}_{-\widehat{\rho }}},\left\Vert
E^{c}\right\Vert _{\mathbb{L}_{\widehat{\rho }_{0}}}\right\},\;\forall 
\mathbf{Z}=\left(E^s,E^u,E^c\right)^T\in\mathcal{X}.
\end{equation*}
The following lemma holds true:

\begin{lemma}
\label{LE2.7} Let $A:D(A)\subset X\to X$ be a closed linear operator and let
us assume that the conditions of Theorem \ref{TH1.4} are satisfied.\ Then
the linear operator $\mathcal{J}$ defined in (\ref{2.49}) satisfies $%
\mathcal{J}\in \mathcal{L}(\mathcal{X})$. More precisely there exists some
constant $C:=C\left( \kappa ,\rho ,\rho _{0},\widehat{\rho },\widehat{\rho }%
_{0}\right) $ such that 
\begin{equation*}
\left\Vert \mathcal{J}(\mathbf{Z})\right\Vert_{\mathcal{X}}\leq C\left\Vert
B\right\Vert _{\mathcal{L}\left( X\right) }\|\mathbf{Z}\|_{\mathcal{X}%
},\;\;\forall \mathbf{Z}\in \mathcal{X}.
\end{equation*}
\end{lemma}

\begin{proof}
Let us notice that $\chi _{+}$ and $\chi _{-}$ are bounded linear operator
defined from $\mathbb{L}_{\widehat{\rho }_{0}}\left( \mathbb{Z},\mathcal{L}%
\left( X\right) \right) $ into $\mathbb{L}_{\widehat{\rho }_{0}}\left( 
\mathbb{N},\mathcal{L}\left( X\right) \right) .$ Furthermore we have%
\begin{equation}
\left\Vert \chi _{+}\right\Vert _{\mathcal{L}\left( \mathbb{L}_{\widehat{%
\rho }_{0}}\left( \mathbb{Z},\mathcal{L}\left( X\right) \right) ,\mathbb{L}_{%
\widehat{\rho }_{0}}\left( \mathbb{N},\mathcal{L}\left( X\right) \right)
\right) }\leq 1,  \label{2.51}
\end{equation}%
and%
\begin{equation}
\left\Vert \chi _{-}\right\Vert _{\mathcal{L}\left( \mathbb{L}_{\widehat{%
\rho }_{0}}\left( \mathbb{Z},\mathcal{L}\left( X\right) \right) ,\mathbb{L}_{%
\widehat{\rho }_{0}}\left( \mathbb{N},\mathcal{L}\left( X\right) \right)
\right) }\leq 1\text{.}  \label{2.52}
\end{equation}%
We also note that $S_{-}\in \mathcal{L}\left( \mathbb{L}_{\widehat{\rho }%
_{0}}\left( \mathbb{Z},\mathcal{L}\left( X\right) \right) \right) $ and $%
S_{-}\in \mathcal{L}\left( \mathbb{L}_{-\widehat{\rho }}\left( \mathbb{N},%
\mathcal{L}\left( X\right) \right) \right) $ with%
\begin{equation}
\left\Vert S_{-}\right\Vert _{\mathcal{L}\left( \mathbb{L}_{\widehat{\rho }%
_{0}}\left( \mathbb{Z},\mathcal{L}\left( X\right) \right) \right) }\leq e^{%
\widehat{\rho }_{0}}\text{ and }\left\Vert S_{-}\right\Vert _{\mathcal{L}%
\left( \mathbb{L}_{-\widehat{\rho }}\left( \mathbb{N},\mathcal{L}\left(
X\right) \right) \right) }\leq 1.  \label{2.53}
\end{equation}%
Therefore by combining (\ref{2.51})-(\ref{2.53}) together with lemmas (\ref%
{LE2.3})-(\ref{LE2.6}) the result follows easily by simple computations.
\end{proof}

As a consequence of the above lemma we obtain the following result:

\begin{proposition}
\label{PR2.9} Let $A:D(A)\subset X\to X$ be given such that the conditions
of Theorem \ref{TH1.4} are satisfied.\ Then there exists $\delta
_{0}:=\delta _{0}\left( \kappa ,\rho ,\rho _{0},\widehat{\rho },\widehat{%
\rho }_{0}\right) \in \left( 0,C^{-1}\right) $ such that for each $\delta
\in \left( 0,\frac{\delta _{0}^{2}}{\kappa +\delta _{0}}\right) $ and each $%
B\in \mathcal{L}(X)$ with $\left\Vert B\right\Vert _{\mathcal{L}\left(
X\right) }\leq \delta$, there exists a unique $\mathbf{Z}=\left(E^s,E^u,E^c%
\right)^T\in \mathcal{X}$ such that (\ref{2.49}) (or equivalently (\ref{2.41}%
)-(\ref{2.43})) holds true. Moreover we have the following properties

\begin{itemize}
\item[(i)] For each $n\in \mathbb{N}$ 
\begin{equation*}
\left\Vert E_{n}^{s}\right\Vert _{\mathcal{L}\left( X\right) }\leq \frac{%
\kappa \delta _{0}}{\delta _{0}-\delta }e^{-\widehat{\rho }n}\text{ and }%
\left\Vert E_{n}^{u}\right\Vert _{\mathcal{L}\left( X\right) }\leq \frac{%
\kappa \delta _{0}}{\delta _{0}-\delta }e^{-\widehat{\rho }n},
\end{equation*}%
and for each $n\in \mathbb{Z}$%
\begin{equation*}
\left\Vert E_{n}^{c}\right\Vert _{\mathcal{L}\left( X\right) }\leq \frac{%
\kappa \delta _{0}}{\delta _{0}-\delta }e^{\widehat{\rho }_{0}\left\vert
n\right\vert }.
\end{equation*}

\item[(ii)] The following estimates hold:%
\begin{equation*}
\left\Vert E_{n}^{s}-A_{s}^{n}\Pi _{s}\right\Vert _{\mathcal{L}\left(
X\right) }\leq \frac{\kappa \delta }{\delta _{0}-\delta }e^{-\widehat{\rho }%
n},\ n\in \mathbb{N}\text{,}
\end{equation*}%
\begin{equation*}
\left\Vert E_{n}^{u}-A_{u}^{-n}\Pi _{u}\right\Vert _{\mathcal{L}\left(
X\right) }\leq \frac{\kappa \delta }{\delta _{0}-\delta }e^{-\widehat{\rho }%
n},\ n\in \mathbb{N}\text{,}
\end{equation*}%
and%
\begin{equation*}
\left\Vert E_{n}^{c}-A_{c}^{n}\Pi _{c}\right\Vert _{\mathcal{L}\left(
X\right) }\leq \frac{\kappa \delta }{\delta _{0}-\delta }e^{\widehat{\rho }%
_{0}\left\vert n\right\vert },\ n\in \mathbb{Z}\text{.}
\end{equation*}

\item[(iii)] One has 
\begin{equation*}
\begin{split}
&E_n^s\left(X\right)\subset D(A),\;\;\forall n\in \mathbb{N}, \\
&E_n^c\left(X\right)\subset D(A),\;\;\forall n\in \mathbb{Z}, \\
&E_n^u\left(X\right)\subset D(A),\;\;\forall n\geq 1\text{ and }\widehat{\Pi}%
_u\left(D(A)\right)\subset D(A).
\end{split}%
\end{equation*}
In particular one has $\widehat{\Pi}_k(X)\subset D(A)$ for $k=s,c$ and for
each $n\geq 0$ $(A+B)\circ E^s_n\in \mathcal{L}(X)$ while for each $n\in%
\mathbb{Z}$ $(A+B)\circ E_n^c\in \mathcal{L}(X)$.
\end{itemize}
\end{proposition}

\begin{proof}
Let $\delta _{0}\in \left( 0,C^{-1}\right) $ be given. Assume that 
\begin{equation}
\left\Vert B\right\Vert _{\mathcal{L}\left( X\right) }\leq \delta \text{
with }\delta \in \left( 0,\frac{\delta _{0}^{2}}{\kappa +\delta _{0}}\right)
.  \label{2.54}
\end{equation}%
Then since $\frac{\delta _{0}^{2}}{\kappa +\delta _{0}}\leq \delta _{0},$
the existence and the uniqueness of a fixed point of (\ref{2.49}) (or
equivalently (\ref{2.41})-(\ref{2.43})) follows from Lemma \ref{LE2.7}.

In the sequel of this proof we denote by $\mathbf{Z}_0=\left(A_{s}^{\cdot
}\Pi _{s},A_{u}^{-\cdot }\Pi _{u},A_{c}^{\cdot }\Pi _{c}\right)^T\in\mathcal{%
X}$ the fixed point of $\mathcal{J}$ with $B=0$. \noindent In order to
obtain the properties \textit{(i)} and \textit{(ii) }we will make use of (%
\ref{2.49}). First of all since $\mathbf{Z}_0\in \mathcal{X}$, let us
observe that using (\ref{1.7})-(\ref{1.9}) we obtain 
\begin{equation}
\left\Vert \mathbf{Z}_0\right\Vert_{\mathcal{X}} \leq \kappa .  \label{2.55}
\end{equation}%
\textbf{Proof of \textit{(i)}:} By using the fixed point problem (\ref{2.49}%
) combined together with Lemma \ref{LE2.7} and (\ref{2.54}) we obtain
(recalling the notation $\mathbf{Z}=\left(E^s,E^u,E^c\right)^T$) that 
\begin{equation*}
\left\Vert \mathbf{Z}\right\Vert_{\mathcal{X}} \leq \kappa +C\delta
\left\Vert \mathbf{Z}\right\Vert_{\mathcal{X}} ,
\end{equation*}%
so that $(i)$ follows from the estimate: 
\begin{equation}
\left\Vert \mathbf{Z}\right\Vert_{\mathcal{X}} \leq \frac{\kappa }{1-C\delta 
}\leq \frac{\kappa \delta _{0}}{\delta _{0}-\delta }.  \label{2.56}
\end{equation}
\noindent \textbf{Proof of \textit{(ii)}:} By using the fixed point problem (%
\ref{2.49}) combined together with Lemma \ref{LE2.7} and (\ref{2.54}) we
obtain that 
\begin{equation}
\left\Vert \mathbf{Z}-\mathbf{Z}_0\right\Vert_{\mathcal{X}} \leq C\delta
\left\Vert \mathbf{Z} \right\Vert_{\mathcal{X}} ,  \label{2.57}
\end{equation}%
so that plugging (\ref{2.56}) into (\ref{2.57}) yields 
\begin{equation*}
\left\Vert \mathbf{Z}-\mathbf{Z}_0\right\Vert \leq \frac{\kappa C\delta }{%
1-C\delta }\leq \frac{\kappa \delta }{\delta _{0}-\delta }.
\end{equation*}%
This prove \textit{(ii)}.\newline
\noindent \textbf{Proof of \textit{(iii)}:} Let us recall that since $%
(A,D(A))$ is exponentially trichotomic (see Definition \ref{DE1.1}) then one
has $D(A)=X_s\oplus X_c\oplus \left(D(A)\cap X_u\right)$. Hence the result
directly follows from right-hand side of \eqref{2.41}-\eqref{2.43}. The
boundedness of $(A+B)\circ E^s_n$ and $(A+B)\circ E^c_n$ follows from the
closed graph theorem since $A$ is closed and $B$ bounded.
\end{proof}

\subsection{Regularized semigroup property and orthogonality property}

\begin{definition}
\label{DE2.10}A family of bounded linear operators $\left\{ W_{n}\right\}
_{n\in \mathbb{N}}\subset \mathcal{L}\left( X\right) $ is a discrete time 
\textbf{regularized semigroup} if%
\begin{equation}
W_{n}W_{p}=W_{n+p},\ \forall n,p\in \mathbb{N}.  \label{2.58}
\end{equation}
\end{definition}

\begin{remark}
\label{RE2.11}If $\left\{ W_{n}\right\} _{n\in \mathbb{N}}\subset \mathcal{L}%
\left( X\right) $ is a regularized semigroup then $W_{0}$ is a bounded
linear projector on $X$.
\end{remark}

\begin{remark}
\label{RE2.12}Observe that if $\left\{ W_{n}\right\} _{n\in \mathbb{N}%
}\subset \mathcal{L}\left( X\right) $ is discrete time regularized semigroup
then by setting $C:=W_{0}$, and using (\ref{2.58}) we obtain $%
W_{n}W_{p}=CW_{n+p}$ for all $n,p\geq 0$. These properties correspond to the
notion of $C$\textbf{-regularized semigroup} given in \cite[Definition 3.1
p.13]{DeLaubenfels} for discrete time.
\end{remark}

In the next lemmas we will show that $\left\{ E_{n}^{k}\right\} _{n\in 
\mathbb{N}}$, $k=s,c,u,$ are regularized semigroup and that we have the
orthogonality property namely for each $n\in \mathbb{N}$ 
\begin{equation*}
E_{n}^{k}E_{n}^{l}=0_{\mathcal{L}\left( X\right) }\text{ if }k,l=s,c,u\text{
with }k\neq l.
\end{equation*}%
The latter equality will allows us to obtain that the bounded linear
projectors $\widehat{\Pi }_{k}=E_{0}^{k},\ k=s,c,u$ satisfy the
orthogonality property%
\begin{equation*}
\widehat{\Pi }_{k}\widehat{\Pi }_{l}=0_{\mathcal{L}\left( X\right) }\text{
if }k,l=s,c,u\text{ with }k\neq l.
\end{equation*}

\begin{lemma}
\label{LE2.13}Let the conditions of Theorem \ref{TH1.4} be satisfied.\ If 
\begin{equation}  \label{cond-B}
\left\Vert B\right\Vert _{\mathcal{L}\left( X\right) }\leq \delta ,\ \text{%
with }\delta \in \left( 0,\frac{\delta _{0}^{2}}{\kappa +\delta _{0}}\right)
\end{equation}%
where $\delta _{0}$ is given in Proposition \ref{PR2.9} then the following
properties hold:

\begin{itemize}
\item[(i)] for each $n,p\in \mathbb{N}$ we have$%
E_{n}^{u}E_{p}^{u}=E_{n+p}^{u}$ and $E_{n}^{s}E_{p}^{u}=0_{\mathcal{L}\left(
X\right) }$, while for each $n\in \mathbb{Z}$,$\ p\in \mathbb{N}$ we have $%
E_{n}^{c}E_{p}^{u}=0_{\mathcal{L}\left( X\right) }$.

\item[(ii)] $\widehat{\Pi }_{u}\in \mathcal{L}(X)$ is a projector on $X$ and
for each $n\geq 0$ one has $E_n^u(X)\subset \widehat{\Pi }_{u}(X)$.
\end{itemize}
\end{lemma}

\begin{proof}
First of all let us notice that since we have $\widehat{\Pi }_{u}=E_{0}^{u}$
the property \textit{(ii)} is a direct consequence of the property\textit{\
(i).} Therefore we will focus on the property \textit{(ii)}.

\noindent The idea of this proof is to derive a suitable closed system of
equations for the following three quantities (wherein $p\in \mathbb{N}$ is
fixed) 
\begin{equation*}
\left\{ E_{n}^{u}E_{p}^{u}-E_{n+p}^{u}\right\} _{n\in \mathbb{N}},\left\{
E_{n}^{s}E_{p}^{u}\right\} _{n\in \mathbb{N}}\ \text{and\ }\left\{
E_{n}^{c}E_{p}^{u}\right\} _{n\in \mathbb{Z}}.
\end{equation*}
Let $p\in \mathbb{N}$ be given and fixed. Then observe that 
\begin{equation*}
\mathbf{W}:=\left(E_{.}^{s}E_{p}^{u},%
\;E_{.}^{u}E_{p}^{u}-E_{.+p}^{u},E_{.}^{c}E_{p}^{u}\right)^T\in\mathcal{X}.
\end{equation*}
\textbf{Equation for }$\left\{ E_{n}^{u}E_{p}^{u}-E_{n+p}^{u}\right\} _{n\in 
\mathbb{N}}$\textbf{:} Let $n\in \mathbb{N}$ and $p\in \mathbb{N}$ be given.
Multiplying the right side of $E_{n}^{u}$ given in (\ref{2.42}) by $%
E_{p}^{u} $ leads us to%
\begin{eqnarray}
E_{n}^{u}E_{p}^{u} &=&A_{u}^{-n}\Pi _{u}E_{p}^{u}+\sum_{m=0}^{+\infty
}A_{u}^{-n-m-1}\Pi _{u}B\left[ E_{m}^{s}E_{p}^{u}+E_{m}^{c}E_{p}^{u}\right]
\label{2.61} \\
&&-\sum_{m=0}^{n-1}\left[ A_{u}^{-m-1}\Pi _{u}\right] BE_{n-m}^{u}E_{p}^{u}+%
\sum_{m=0}^{+\infty }\left[ A_{s}^{m}\Pi _{s}+A_{c}^{m}\Pi _{c}\right]
BE_{n+m+1}^{u}E_{p}^{u}.  \notag
\end{eqnarray}%
Next by using also (\ref{2.42}) and replacing $n$ with $n+p$ we obtain%
\begin{eqnarray}
E_{n+p}^{u} &=&A_{u}^{-n-p}\Pi _{u}+\sum_{m=0}^{+\infty }A_{u}^{-n-p-m-1}\Pi
_{u}B\left[ E_{m}^{s}+E_{m}^{c}\right]  \label{2.62} \\
&&-\sum_{m=0}^{n+p-1}\left[ A_{u}^{-m-1}\Pi _{u}\right] BE_{n+p-m}^{u}+%
\sum_{m=0}^{+\infty }\left[ A_{s}^{m}\Pi _{s}+A_{c}^{m}\Pi _{c}\right]
BE_{n+p+m+1}^{u}.  \notag
\end{eqnarray}%
Therefore by subtracting (\ref{2.62}) from (\ref{2.61}) we get%
\begin{eqnarray}
E_{n}^{u}E_{p}^{u}-E_{n+p}^{u} &=&A_{u}^{-n}\Pi
_{u}E_{p}^{u}-A_{u}^{-n-p}\Pi _{u}  \label{2.63} \\
&&-\sum_{m=0}^{+\infty }A_{u}^{-n-p-m-1}\Pi _{u}B\left[ E_{m}^{s}+E_{m}^{c}%
\right]  \notag \\
&&+\sum_{m=0}^{n+p-1}\left[ A_{u}^{-m-1}\Pi _{u}\right] BE_{n+p-m}^{u}-%
\sum_{m=0}^{n-1}\left[ A_{u}^{-m-1}\Pi _{u}\right] BE_{n-m}^{u}E_{p}^{u} 
\notag \\
&&+\sum_{m=0}^{+\infty }A_{u}^{-n-m-1}\Pi _{u}B\left[
E_{m}^{s}E_{p}^{u}+E_{m}^{c}E_{p}^{u}\right]  \notag \\
&&+\sum_{m=0}^{+\infty }\left[ A_{s}^{m}\Pi _{s}+A_{c}^{m}\Pi _{c}\right] B%
\left[ E_{n+m+1}^{u}E_{p}^{u}-E_{n+p-m}^{u}\right] .  \notag
\end{eqnarray}%
Now note that by using (\ref{2.42}), replacing $n$ with $p$ in order to
obtain $E_{p}^{u}$ and multiply its left side by $A_{u}^{-n}\Pi _{u}$ we
obtain 
\begin{eqnarray*}
A_{u}^{-n}\Pi _{u}E_{p}^{u} &=&A_{u}^{-n-p}\Pi _{u}+\sum_{m=0}^{+\infty
}A_{u}^{-n-p-m-1}\Pi _{u}B\left[ E_{m}^{s}+E_{m}^{c}\right] \\
&&-\sum_{m=0}^{p-1}\left[ A_{u}^{-n-m-1}\Pi _{u}\right] BE_{p-m}^{u},
\end{eqnarray*}%
and since we have%
\begin{equation*}
-\sum_{m=0}^{p-1}\left[ A_{u}^{-n-m-1}\Pi _{u}\right] BE_{p-m}^{u}=-%
\sum_{m=n}^{n+p-1}\left[ A_{u}^{-m-1}\Pi _{u}\right] BE_{n+p-m}^{u},
\end{equation*}%
it follows that%
\begin{eqnarray}
A_{u}^{-n}\Pi _{u}E_{p}^{u} &=&A_{u}^{-n-p}\Pi _{u}+\sum_{m=0}^{+\infty
}A_{u}^{-n-p-m-1}\Pi _{u}B\left[ E_{m}^{s}+E_{m}^{c}\right]  \label{2.64} \\
&&-\sum_{m=n}^{n+p-1}\left[ A_{u}^{-m-1}\Pi _{u}\right] BE_{n+p-m}^{u}. 
\notag
\end{eqnarray}%
Therefore by plugging the expression of $A_{u}^{-n}\Pi _{u}E_{p}^{u}$ given
by (\ref{2.64}) into (\ref{2.63}) and recalling \eqref{2.49} we obtain that%
\begin{equation}  \label{2.65}
E_{.}^{u}E_{p}^{u}-E_{.+p}^{u}=\left(\mathcal{J}_{21},\mathcal{J}_{22},%
\mathcal{J}_{23}\right)\mathbf{W}.
\end{equation}
\textbf{Equation for }$\left\{ E_{n}^{s}E_{p}^{u}\right\} _{n\in \mathbb{N}}$%
\textbf{:} Let $n\in \mathbb{N}$ and $p\in \mathbb{N}$ be given. Then by
using (\ref{2.41}) and multiply the right side of $E_{n}^{s}$ by $E_{p}^{u}$
we obtain%
\begin{eqnarray}
E_{n}^{s}E_{p}^{u} &=&A_{s}^{n}\Pi _{s}E_{p}^{u}-\sum_{m=0}^{+\infty
}A_{s}^{m+n}\Pi _{s}B\left[ E_{m+1}^{u}E_{p}^{u}+E_{-m-1}^{c}E_{p}^{u}\right]
\label{2.66} \\
&&+\sum_{m=0}^{n-1}A_{s}^{n-m-1}\Pi
_{s}BE_{m}^{s}E_{p}^{u}-\sum_{m=0}^{+\infty }\left[ A_{u}^{-m-1}\Pi
_{u}+A_{c}^{-m-1}\Pi _{c}\right] BE_{n+m}^{s}E_{p}^{u}.  \notag
\end{eqnarray}%
Next by replacing $n$ by $p$ in (\ref{2.42}) we obtain $E_{p}^{u}$ and by
multiplying its left side by $A_{s}^{n}\Pi _{s}$ we get%
\begin{equation}
A_{s}^{n}\Pi _{s}E_{p}^{u}=\sum_{m=0}^{+\infty }A_{s}^{n+m}\Pi
_{s}E_{p+m+1}^{u}.  \label{2.67}
\end{equation}%
Then plugging (\ref{2.67}) into (\ref{2.66}) yields%
\begin{equation}  \label{2.68}
E_{.}^{s}E_{p}^{u}=\left(\mathcal{J}_{11},\mathcal{J}_{12},\mathcal{J}%
_{13}\right)\mathbf{W}.
\end{equation}
\textbf{Equation for }$\left\{ E_{n}^{c}E_{p}^{u}\right\} _{n\in \mathbb{Z}}$%
:\textbf{\ }Let $n\in \mathbb{Z}$ and $p\in \mathbb{N\ }$be given. By
multiplying the right side of (\ref{2.43}) by $E_{p}^{u}$ we get 
\begin{eqnarray}
E_{n}^{c}E_{p}^{u} &=&A_{c}^{n}\Pi _{c}E_{p}^{u}-\sum_{m=0}^{+\infty
}A_{c}^{m+n}\Pi _{c}BE_{m+1}^{u}E_{p}^{u}+\sum_{m=0}^{+\infty
}A_{c}^{n-m-1}\Pi _{c}BE_{m}^{s}E_{p}^{u}  \notag \\
&&+\sum_{m=0}^{n-1}A_{c}^{n-m-1}\Pi
_{c}BE_{m}^{c}E_{p}^{u}-\sum_{m=0}^{-n-1}A_{c}^{-m-1}\Pi
_{c}BE_{n+m}^{c}E_{p}^{u}  \label{2.69} \\
&&-\sum_{m=0}^{+\infty }A_{u}^{-m-1}\Pi
_{u}BE_{n+m}^{c}E_{p}^{u}+\sum_{m=0}^{+\infty }A_{s}^{m}\Pi
_{s}BE_{-m-1+n}^{c}E_{p}^{u}.  \notag
\end{eqnarray}%
Next by replacing $n$ by $p$ in (\ref{2.42}) we obatin $E_{p}^{u}$ and by
multiplying its left side by $A_{c}^{n}\Pi _{c}$ we get 
\begin{equation}
A_{c}^{n}\Pi _{c}E_{p}^{u}=\sum_{m=0}^{+\infty }A_{c}^{n+m}\Pi
_{c}E_{p+m+1}^{u},  \label{2.70}
\end{equation}%
Therefore by plugging (\ref{2.70}) into (\ref{2.69}) we get%
\begin{equation}  \label{2.71}
E_{.}^{c}E_{p}^{u}=\left(\mathcal{J}_{31},\mathcal{J}_{32},\mathcal{J}%
_{33}\right)\mathbf{W}.
\end{equation}

Recalling \eqref{2.49}, it follows that $\mathbf{W}$ satisfies $\mathbf{W}=%
\mathcal{J}(\mathbf{W})$. Hence we infer from Lemma \ref{LE2.7} that since $%
C\left\Vert B\right\Vert _{\mathcal{L}\left( X\right)}\leq C\delta <1$, one
has $\mathbf{W}=0_{\mathcal{X}}$ and this completes the proof of the lemma. 
%
%
\end{proof}

\bigskip

\begin{remark}
\label{RE2.14}The arguments for the proof of the next two lemmas are similar
to the arguments used for the proof of Lemma \ref{LE2.13}.
\end{remark}

\begin{lemma}
\label{LE2.15}Let the conditions of Theorem \ref{TH1.4} and \eqref{cond-B}
be satisfied. Then the following properties hold true:

\begin{itemize}
\item[(i)] for each $n,p\in \mathbb{N}$ we have $%
E_{n}^{s}E_{p}^{s}=E_{n+p}^{s}$ and $E_{n}^{u}E_{p}^{s}=0_{\mathcal{L}\left(
X\right) }$ while for each $n\in \mathbb{Z}$,$\ p\in \mathbb{N}$ we have $%
E_{n}^{c}E_{p}^{s}=0_{\mathcal{L}\left( X\right) }$.

\item[(ii)] $\widehat{\Pi }_{s}$ is a bounded linear projector on $X$ and
for each $n\geq 0$ one has $E_n^s(X)\subset \widehat{\Pi }_{s}(X)$.
\end{itemize}
\end{lemma}

\begin{proof}
First of all let us notice that since we have $\widehat{\Pi }_{s}=E_{0}^{s}$
the property \textit{(ii) }is a direct consequence of the property\textit{\
(i). }Therefore we will focus on the property \textit{(ii)}.

\noindent The idea of this proof is to derive a suitable closed system of
equations for the following three quantities (wherein $p\in \mathbb{N}$ is
fixed): 
\begin{equation*}
\left\{ E_{n}^{s}E_{p}^{s}-E_{n+p}^{s}\right\} _{n\in \mathbb{N}},\left\{
E_{n}^{u}E_{p}^{s}\right\} _{n\in \mathbb{N}}\ \text{and\ }\left\{
E_{n}^{c}E_{p}^{s}\right\} _{n\in \mathbb{Z}}.
\end{equation*}%
Let $p\in\mathbb{N}$ be given and fixed and let us observe that 
\begin{equation*}
\mathbf{W}:=%
\left(E_{.}^{s}E_{p}^{s}-E_{.+p}^{s},E_{.}^{u}E_{p}^{s},E_{.}^{c}E_{p}^{s}%
\right)^T\in\mathcal{X}.
\end{equation*}
By proceeding as in the proof of Lemma \ref{LE2.13} we obtain the following
closed system of equations $\mathbf{W}=\mathcal{J}\left(\mathbf{W}\right)$,
that ensures that $\mathbf{W}=0_{\mathcal{X}}$. This ends the proof of this
lemma. %
\end{proof}

\begin{lemma}
\label{LE2.16}Let the conditions of Theorem \ref{TH1.4} and \eqref{cond-B}
be satisfied. Then the following properties hold true:

\begin{itemize}
\item[(i)] for each $n,p\in \mathbb{Z}$ we have $%
E_{n}^{c}E_{p}^{c}=E_{n+p}^{c}$ and for each $n\in \mathbb{N}$,$\ p\in 
\mathbb{Z}$ we have $E_{n}^{s}E_{p}^{c}=E_{n}^{u}E_{p}^{c}=0_{\mathcal{L}%
\left( X\right) }$.

\item[(ii)] $\widehat{\Pi }_{c}$ is a bounded linear projector on $X$ and
for each $n\in\mathbb{Z}$ one has $E_n^c(X)\subset \widehat{\Pi }_{c}(X)$.
\end{itemize}
\end{lemma}

\begin{proof}
First of all let us notice that since we have $\widehat{\Pi }_{c}=E_{0}^{c}$
the property \textit{(ii) }is a direct consequence of the property\textit{\
(i). }Therefore we will focus on the property \textit{(ii)}.

\noindent The idea of this proof is to derive a suitable closed system of
equations for the following three quantities (wherein $p\in \mathbb{N}$ is
fixed): 
\begin{equation*}
\left\{ E_{n}^{c}E_{p}^{c}-E_{n+p}^{c}\right\} _{n\in \mathbb{Z}},\left\{
E_{n}^{u}E_{p}^{c}\right\} _{n\in \mathbb{N}}\ \text{and\ }\left\{
E_{n}^{s}E_{p}^{c}\right\} _{n\in \mathbb{N}}.
\end{equation*}
Let $p\in\mathbb{Z}$ be given and fixed and observe that: 
\begin{equation*}
\mathbf{W}:=%
\left(E_{.}^{s}E_{p}^{c},E_{.}^{u}E_{p}^{c},E_{.}^{c}E_{p}^{c}-E_{.+p}^{c}%
\right)^T\in\mathcal{X}.
\end{equation*}
By proceeding as in the proof of Lemma \ref{LE2.13} we obtain the following
closed system of equations $\mathbf{W}=\mathcal{J}\left(\mathbf{W}\right)$.
This completes the proof of this lemma. 
\end{proof}


\subsection{Proof of Theorem \protect\ref{TH1.4}}

In this section we complete the proof of Theorem \ref{TH1.4}. The main
points are summarized in the following lemma. Note that the proof of Theorem %
\ref{TH1.4} becomes a direct consequence of Proposition \ref{PR2.9} and
Lemma \ref{LE2.17} below.

\begin{lemma}
\label{LE2.17}Let us assume that the conditions of Theorem \ref{TH1.4} are
satisfied. Up to reduce the value of $\delta _{0}$ provided by Proposition %
\ref{PR2.9} so that $\delta _{0}<\min \left(C^{-1}, \frac{1-C^{-1}}{6\kappa
^{3}+1-C^{-1}},,\sqrt{2}-1\right) $, if $B\in \mathcal{L}(X)$ satisfies 
\begin{equation*}
\left\Vert B\right\Vert _{\mathcal{L}\left( X\right) }\leq \delta ,\ \text{%
with }\delta \in \left( 0,\frac{\delta _{0}^{2}}{\kappa +\delta _{0}}\right)
\end{equation*}%
then the following properties hold:

\begin{itemize}
\item[(i)] The three bounded linear projectors $\widehat{\Pi }_{s},\widehat{%
\Pi }_{u}$ and $\widehat{\Pi }_{c}$ provided by Lemmas \ref{LE2.13}, \ref%
{LE2.15} and \ref{LE2.16} satisfy%
\begin{equation}
\widehat{\Pi }_{k}\widehat{\Pi }_{l}=0_{\mathcal{L}\left( X\right) }\text{
if }k\neq l,\ \text{with }k,l=s,u,c,  \label{2.78}
\end{equation}%
and%
\begin{equation}
\left\Vert \widehat{\Pi }_{k}-\Pi _{k}\right\Vert _{\mathcal{L}\left(
X\right) }\leq \frac{\kappa \delta }{\delta _{0}-\delta }\leq \delta _{0}.
\label{2.79}
\end{equation}

\item[(ii)] For each $n\in \mathbb{N}$ and $k=s,c$ we have $E_{n}^{k}=\left(
A+B\right) ^{n}\widehat{\Pi }_{k}\in \mathcal{L}\left(X,\widehat{\Pi }%
_{k}(X)\right).$

\item[(iii)] For each $n\in \mathbb{N}$ $\left( A+B\right) ^{n}\widehat{\Pi }%
_{c}$ is invertible from $\widehat{\Pi }_{c}\left( X\right) $ into $\widehat{%
\Pi }_{c}\left( X\right) $ with%
\begin{equation}
E_{-n}^{c}\left( A+B\right) ^{n}\widehat{\Pi }_{c}=\left( A+B\right) ^{n}%
\widehat{\Pi }_{c}E_{-n}^{c}=\widehat{\Pi }_{c}.  \label{2.80}
\end{equation}

\item[(iv)] One has $\left(A+B\right)\left(D(A)\cap \widehat{\Pi}%
_u(X)\right)\subset \widehat{\Pi}_u(X)$. Consider $(A+B)_u:D(A)\cap \widehat{%
\Pi}_u(X)\subset \widehat{\Pi}_u(X)\to \widehat{\Pi}_u(X)$ the part of $%
(A+B) $ in $\widehat{\Pi}_u(X)$. Then one has $0\in \rho\left(
(A+B)_u\right) $ and for each $n\geq 0$: 
\begin{equation}
E_{n}^{u}=\left( (A+B)_u\right) ^{-n}\widehat{\Pi }_{u}.  \label{2.81}
\end{equation}

\item[(v)] For $k=s,u,c$, the projector $\widehat{\Pi }_{k}$ satisfies $%
\widehat{\Pi }_{k}\left(D(A)\right)\subset D(A)$ and 
\begin{equation}
\left( A+B\right) \widehat{\Pi }_{k}x=\widehat{\Pi }_{k}\left(
A+B\right)x,\;\forall x\in D(A) .  \label{2.82}
\end{equation}
\end{itemize}
\end{lemma}

\begin{proof}
\textbf{Proof of \textit{(i)}:} By recalling that $E_{0}^{k}=\widehat{\Pi }%
_{k}$ it follows from Lemmas \ref{LE2.13}, \ref{LE2.15} and \ref{LE2.16}
that (\ref{2.78}) holds true. Moreover the condition (ii) of Proposition \ref%
{PR2.9} together with $\delta \in \left( 0,\frac{\delta _{0}^{2}}{\kappa
+\delta _{0}}\right) $ provide that%
\begin{equation}
\left\Vert \widehat{\Pi }_{k}-\Pi _{k}\right\Vert _{\mathcal{L}\left(
X\right) }\leq \frac{\kappa \delta }{\delta _{0}-\delta }\leq \delta _{0}\in
\left( 0,\sqrt{2}-1\right) .  \label{2.83}
\end{equation}%
This completes the proof of \textit{(i).}

\noindent \textbf{Proof of \textit{(ii)}: }Let\textbf{\ }$n\in \mathbb{N}%
\backslash \left\{ 0\right\} $ be given. We will first prove that $%
E_{n}^{s}=\left( A+B\right) ^{n}\widehat{\Pi }_{s}$. By replacing $n$ by $%
n-1 $ in (\ref{2.41}), recalling that $E_{n-1}^s(X)\subset \widehat{\Pi}%
_s(X)\cap D(A)$ (see Proposition \ref{PR2.9} $(iii)$) and multiplying the
left side of $E_{n-1}^{s}$ by $A$ it follows that%
\begin{equation}
AE_{n-1}^{s}=E_{n}^{s}-BE_{n-1}^{s}\Longleftrightarrow E_{n}^{s}=\left(
A+B\right) E_{n-1}^{s}.  \notag
\end{equation}
Hence by induction (see Proposition \ref{PR2.9} $(iii)$) one obtains that
for each $n\geq 0$: $(A+B)^n\widehat{\Pi}_s(X)\subset D(A)$, $(A+B)^n%
\widehat{\Pi}_s\in \mathcal{L}\left(X,\widehat{\Pi}_s(X)\right)$ and 
\begin{equation*}
E_{n}^{s}=\left( A+B\right) ^{n}E_{0}^{s}=\left( A+B\right) ^{n}\widehat{\Pi 
}_{s}.
\end{equation*}%
Next we prove that $E_{n}^{c}=\left( A+B\right) ^{n}\widehat{\Pi }_{c}$ for
each $n\in \mathbb{N}$. Let $n\in \mathbb{N}\backslash \left\{ 0\right\} $
be given. By replacing $n$ by $n-1$ in (\ref{2.43}), recalling that $%
E_{n-1}^c(X)\subset D(A)$ and multiplying the left side of $E_{n-1}^{c}$ by $%
A$ we obtain%
\begin{equation*}
AE_{n-1}^{c}=E_{n}^{c}-BE_{n-1}^{c}\Longleftrightarrow E_{n}^{c}=\left(
A+B\right) E_{n-1}^{c},
\end{equation*}%
providing that%
\begin{equation}
E_{n}^{c}=\left( A+B\right) ^{n}\widehat{\Pi }_{c}\in \mathcal{L}\left(X,%
\widehat{\Pi}_c(X)\right).  \label{2.84}
\end{equation}%
This completes the proof of \textit{(ii)}.

\noindent \textbf{Proof of \textit{(iii)}: } Let us prove that for each $%
n\in \mathbb{N}$ the bounded linear operator $\left( A+B\right) ^{n}\widehat{%
\Pi }_{c}$ is invertible from $\widehat{\Pi }_{c}\left( X\right) $ into $%
\widehat{\Pi }_{c}\left( X\right) .$

In fact each $n\in \mathbb{N}$ by using Lemma \ref{LE2.16} combined together
with (\ref{2.84}) we obtain%
\begin{equation*}
E_{-n}^{c}\left( A+B\right) ^{n}\widehat{\Pi }%
_{c}=E_{-n}^{c}E_{n}^{c}=E_{0}^{c}=\widehat{\Pi }_{c}=E_{n}^{c}E_{-n}^{c}=%
\left( A+B\right) ^{n}\widehat{\Pi }_{c}E_{-n}^{c}.
\end{equation*}%
This prove that $\left( A+B\right) ^{n}\widehat{\Pi }_{c}$ is invertible
from $\widehat{\Pi }_{c}\left( X\right) $ into $\widehat{\Pi }_{c}\left(
X\right) $ and (\ref{2.80}) holds true.

\noindent \textbf{Proof of \textit{(iv)}: } In order to prove this point we
claim that

\begin{claim}
\label{claim++} The following holds true:

\begin{itemize}
\item[(a)] Recalling that $E_1^u(X)\subset D(A)$ one has $(A+B)E_1^u=%
\widehat{\Pi}_u$.

\item[(b)] Consider the closed linear operator $C_u:D(C_u)\subset \widehat{%
\Pi}_u(X)\to \widehat{\Pi}_u(X)$ defined by $D\left(C_u\right)=D(A)\cap 
\widehat{\Pi}_u(X)$ and $C_u =\widehat{\Pi}_u\left(A+B\right)$. Then it
satisfies $0\in \rho\left(C_u\right)$.
\end{itemize}
\end{claim}

Before proving this claim let us complete the proof of \eqref{2.81}. To do
so let us first notice that $(a)$ and $(b)$ implies that 
\begin{equation*}
\left(A+B\right)\left(D(A)\cap \widehat{\Pi}_u(X)\right)=(A+B)\left(C_u^{-1}%
\left(\widehat{\Pi}_u(X)\right)\right)=\widehat{\Pi}_u(X).
\end{equation*}
Hence the linear operator $\left(C_u,D(C_u)\right)$ coincide the part $%
\left(A+B\right)_u$ of $(A+B)$ in $\widehat{\Pi}_u(X)$. Therefore $0\in
\rho\left(\left(A+B\right)_u\right)$ and using $(a)$ and the orthogonality
of the perturbed projectors one gets: 
\begin{equation*}
E_1^u=\left(\left(A+B\right)_u\right)^{-1}\widehat{\Pi}_u.
\end{equation*}
Finally due to the semiflow property for $E_n^u$ one gets 
\begin{equation*}
E_n^u=\left(\left(A+B\right)_u\right)^{-n}\widehat{\Pi}_u,\;\;n\geq 0,
\end{equation*}
and \eqref{2.81} follows.

It remains to prove Claim \ref{claim++}.

\noindent \textbf{Proof of $(a)$:} Let us first recall that $E_1(X)\subset
D(A)$ and let us multiply the left side of $E_{1}^{u}$ given in (\ref{2.42})
by $A$ to obtain%
\begin{equation*}
AE_{1}^{u}=E_{0}^{u}-BE_{1}^{u}\Longleftrightarrow \left( A+B\right)
E_{1}^{u}=E_{0}^{u},
\end{equation*}
that completes the proof of $(a)$.

\noindent \textbf{Proof of $(b)$:} Before proceeding to the proof of this
statement let us notice that since we have $E_{0}^{k}=\widehat{\Pi }_{k},$ $%
k=s,u,c$ it follows from the condition \textit{(i)} of Proposition \ref%
{PR2.9} that 
\begin{equation}
\left\Vert \widehat{\Pi }_{k}\right\Vert _{\mathcal{L}\left( X\right) }\leq 
\frac{\kappa \delta _{0}}{\delta _{0}-\delta },\ k=s,u,c.  \label{2.89}
\end{equation}%
Now recalling that $D(A)=X_s\oplus X_c\oplus \left(X_u\cap D(A)\right)$ one
has that for each $x\in D(A)\cap \widehat{\Pi}_u(X)$: 
\begin{equation*}
x=\Pi_s x+\Pi_c x+\Pi_u x,
\end{equation*}
so that $\Pi_u x\in X_u\cap D(A)$. This re-writes as $\Pi_u\left(D(A)\cap 
\widehat{\Pi}_u(X)\right)\subset D(A)\cap \Pi_u(X)$. Hence one has 
\begin{equation*}
C_u=\widehat{\Pi }_{u}A\left[ \Pi _{u}+\Pi _{s}+\Pi _{c}\right] +\widehat{%
\Pi }_{u}B.
\end{equation*}
This re-writes as 
\begin{equation}  \label{eq-pert}
C_u=\widetilde A_u+ L_u,
\end{equation}
wherein we have set $\widetilde A_u:D(C_u)\subset \widehat{\Pi}_u(X)\to 
\widehat{\Pi}_u(X)$ defined as 
\begin{equation*}
\widetilde A_u x=\widehat{\Pi }_{u}A\Pi_u x,\;\;\forall x\in D(A)\cap 
\widehat{\Pi}_u(X),
\end{equation*}
and $L_u\in \mathcal{L}\left(\widehat{\Pi}_u(X)\right)$ defined by: 
\begin{equation*}
L_{u}:=\widehat{\Pi }_{u}A_s\Pi _{s}\widehat{\Pi }_{u}+\widehat{\Pi }%
_{u}A_c\Pi _{c}\widehat{\Pi }_{u}+\widehat{\Pi }_{u}B\widehat{\Pi }_{u}.
\end{equation*}%
Next observe that due to (\ref{2.83}) Lemma \ref{LE2.1} applies to $\Pi _{u}$
and $\widehat{\Pi }_{u}$ and provides that $\Pi _{u}|_{\widehat{\Pi }%
_{u}\left( X\right) }$ is an isomorphism from $\widehat{\Pi }_{u}\left(
X\right) $ onto $\Pi _{u}\left( X\right) $ while $\widehat{\Pi }_{u}|_{\Pi
_{u}\left( X\right) }$ is an isomorphism from $\Pi _{u}\left( X\right) $
onto $\widehat{\Pi }_{u}\left( X\right) $. One furthermore has 
\begin{equation}  \label{isom1}
\left\Vert \left( \widehat{\Pi }_{u}|_{\Pi _{u}\left( X\right) }\right)
^{-1}x\right\Vert \leq \frac{1}{1-\delta _{0}}\left\Vert x\right\Vert ,\
\forall x\in \Pi _{u}\left( X\right) ,
\end{equation}%
and%
\begin{equation}  \label{isom2}
\left\Vert \left( \Pi _{u}|_{\widehat{\Pi }_{u}\left( X\right) }\right)
^{-1}x\right\Vert \leq \frac{1}{1-\delta _{0}}\left\Vert x\right\Vert ,\
\forall x\in \widehat{\Pi }_{u}\left( X\right) .
\end{equation}%
Then due to the above isomorphism one has 
\begin{equation*}
\Pi_u\left(D(A)\cap \widehat{\Pi}_u(X)\right)=D(A)\cap \Pi_u(X).
\end{equation*}
Indeed first note that inclusion $\subset$ has already been observed.
Consider $x\in D(A)\cap \Pi_u(X)$. Then there exists a unique $y\in \widehat{%
\Pi}_u(X)$ such that $\Pi_u(y)=x$. Then we write $y=\Pi_s y+ \Pi_c y+\Pi_u y$%
. Since $D(A)=X_s\oplus X_c\oplus \left(D(A)\cap X_u\right)$ and $\Pi_u
y=x\in D(A)\cap \Pi_u(X)$ one obtains that $y\in D(A)\cap \widehat{\Pi}_u(X)$
and $x\in \Pi_u\left(D(A)\cap \widehat{\Pi}_u(X)\right)$ and the equality
follows.

As a consequence one gets: 
\begin{equation*}
D\left(C_u\right)=D(A)\cap \widehat{\Pi}_u(X)=\left( \Pi _{u}|_{\widehat{\Pi 
}_{u}\left( X\right) }\right) ^{-1}\left(D(A)\cap \Pi_u(X)\right).
\end{equation*}
Using this relation and recalling that $0\in \rho\left(A_u\right)$ it is
easy to check that $0\in \rho\left(\widetilde{A}_u\right)$ and 
\begin{equation*}
\left(\widetilde{A}_u\right)^{-1}=\left(\Pi_{u}|_{\widehat{\Pi }_{u}\left(
X\right) }\right) ^{-1}\circ A_u^{-1} \circ \left( \widehat{\Pi }%
_{u}|_{\Pi_{u}(X)}\right)^{-1}.
\end{equation*}

Finally due to \eqref{eq-pert}, in order to complete the proof of point $(b)$
it is sufficient to check that 
\begin{equation*}
\left\Vert L_u\right\Vert _{\mathcal{L}\left( \widehat{\Pi }_{u}\left(
X\right) \right) }\left\Vert \widetilde A_u^{-1}\right\Vert_{\mathcal{L}%
\left( \widehat{\Pi }_{u}\left( X\right) \right) }<1.
\end{equation*}
To do so let us first notice that due to \eqref{isom1}-\eqref{isom2} one has 
\begin{equation}  \label{2.90}
\left\Vert \widetilde A_u^{-1}\right\Vert_{\mathcal{L}\left(\widehat{\Pi }%
_{u}\left( X\right)\right)} \leq \frac{1}{\left( 1-\delta _{0}\right)^2 }%
\left\Vert A_{u}^{-1}\right\Vert _{\mathcal{L}\left( \Pi_u(X)\right) }\leq 
\frac{1}{\left( 1-\delta _{0}\right)^2 }\kappa e^{-\rho }.
\end{equation}
On the other one has 
\begin{eqnarray*}
L_{u} &=&\widehat{\Pi }_{u}A\Pi _{s}\widehat{\Pi }_{u}+\widehat{\Pi }%
_{u}A\Pi _{c}\widehat{\Pi }_{u}+\widehat{\Pi }_{u}B\widehat{\Pi }_{u} \\
&=&\widehat{\Pi }_{u}A_{s}\Pi _{s}\left[ \widehat{\Pi }_{u}-\Pi _{u}\right] +%
\widehat{\Pi }_{u}A_{c}\Pi _{c}\left[ \widehat{\Pi }_{u}-\Pi _{u}\right] +%
\widehat{\Pi }_{u}B\widehat{\Pi }_{u}.
\end{eqnarray*}%
Then by using (\ref{2.79}) and (\ref{2.89}) and recalling that $\left\Vert
B\right\Vert _{\mathcal{L}\left( X\right) }\leq \delta _{0}$, it follows
that 
\begin{equation}  \label{2.91}
\left\Vert L_{u}\right\Vert _{\mathcal{L}\left( \widehat{\Pi }_{u}\left(
X\right) \right) } \leq 2\kappa ^{2}\delta _{0}+2\kappa ^{2}e^{\rho
_{0}}+2\kappa \delta _{0} \leq 6\kappa ^{2}e^{\rho _{0}}\delta _{0}.
\end{equation}%
Now combining (\ref{2.90}) together with (\ref{2.91}) provides that 
\begin{equation*}
\left\Vert L_u\right\Vert _{\mathcal{L}\left( \widehat{\Pi }_{u}\left(
X\right) \right) }\left\Vert \widetilde A_u^{-1}\right\Vert_{\mathcal{L}%
\left( \widehat{\Pi }_{u}\left( X\right) \right) } \leq \frac{6\kappa ^{3}}{%
\left( 1-\delta _{0}\right)^2 }\delta _{0}.
\end{equation*}
Hence up to reduce $\delta _{0}$ such that 
\begin{equation}
\delta _{0}<\min \left\{ C^{-1},\frac{1-C^{-1}}{6\kappa ^{3}+1-C^{-1}}%
\right\} ,  \label{2.92}
\end{equation}%
where $C>1$ is the constant provided by Lemma \ref{LE2.7} we obtain that 
\begin{equation*}
\left\Vert L_u\right\Vert _{\mathcal{L}\left( \widehat{\Pi }_{u}\left(
X\right) \right) }\left\Vert \widetilde A_u^{-1}\right\Vert_{\mathcal{L}%
\left( \widehat{\Pi }_{u}\left( X\right) \right) }<1.
\end{equation*}
This completes the proof of Claim \ref{claim++} $(b)$ and also the proof of 
\textit{(iv).}

\noindent \textbf{Proof of \textit{(v)}:} Let us first notice that the
inclusions $\widehat{\Pi}_k(X)\subset D(A)$ for any $k=s,c$ and $\widehat{\Pi%
}_u(D(A))\subset D(A)$ have been observed in Proposition \ref{PR2.9} $(iii)$%
. Next recall that by \textit{(ii)} we have 
\begin{equation*}
E_{1}^{k}=\left( A+B\right) \widehat{\Pi }_{k},\ \forall k=s,c,
\end{equation*}%
so that 
\begin{equation*}
\widehat{\Pi }_{k}\left( A+B\right) \widehat{\Pi }_{k}=\widehat{\Pi }%
_{k}E_{1}^{k}=E_{0}^{k}E_{1}^{k}=E_{1}^{k}=\left( A+B\right) \widehat{\Pi }%
_{k},
\end{equation*}%
that is 
\begin{equation}
\left( A+B\right) \widehat{\Pi }_{k}=\widehat{\Pi }_{k}\left( A+B\right) 
\widehat{\Pi }_{k},\ k=s,c.  \label{2.93}
\end{equation}%
Moreover the property \textit{(iii)} implies that $\left( A+B\right) 
\widehat{\Pi }_{u}$ maps $D(A)$ into $\widehat{\Pi }_{u}$, that is for any $%
x\in D(A)$: 
\begin{equation}
\left( A+B\right) \widehat{\Pi }_{u}x=\widehat{\Pi }_{u}\left( A+B\right) 
\widehat{\Pi }_{u}x.  \label{2.94}
\end{equation}%
Therefore for each $k=s,c,u$ by using (\ref{2.93}) and (\ref{2.94}) combined
together with the orthogonality property in (\ref{2.78}) we obtain for each $%
x\in D(A)$: 
\begin{eqnarray*}
\widehat{\Pi }_{k}\left( A+B\right)x &=&\widehat{\Pi }_{k}\left( A+B\right) %
\left[ \widehat{\Pi }_{s}+\widehat{\Pi }_{u}+\widehat{\Pi }_{c}\right]x \\
&=&\widehat{\Pi }_{k}\left( A+B\right) \widehat{\Pi }_{k}x=\left( A+B\right) 
\widehat{\Pi }_{k}x.
\end{eqnarray*}%
This completes the proof of this lemma
\end{proof}

\section{Proof of Theorem \protect\ref{TH1.8}}

The aim of this section is to complete the proof of Theorem \ref{TH1.8}.

Let $q\in [1,\infty]$ be given. Recall that we denote the Banach space $%
X=l^{q}(\mathbb{Z};Y)$. Recall also the definition of the linear operator $%
\left(\mathcal{A},D\left(\mathcal{A}\right)\right)$ in \eqref{operateur_A}.
Next let us consider the three bounded linear operators $\mathcal{P}_{\alpha
}\in \mathcal{L}(X)$ defined for $\alpha =s,c,u$ by 
\begin{equation*}
\left( \mathcal{P}_{\alpha }u\right) _{k}=\Pi _{k}^{\alpha
}u_{k},\;\;\forall k\in \mathbb{Z},\;\forall u\in X.
\end{equation*}%
Using the above notations let us notice that for each $\alpha =s,c,u$, $%
\mathcal{P}_{\alpha }$ is a projector on $X$ that satisfies

\begin{itemize}
\item $\mathcal{P}_{\alpha }\mathcal{P}_{\beta }=0_{\mathcal{L}(X)}$ for all 
$\alpha \neq \beta $.

\item $\mathcal{P}_{s}+\mathcal{P}_{c}+\mathcal{P}_{u}=I_{\mathcal{L}(X)}$.

\item for each $\alpha =s,c,u$, one has $\mathcal{A}\left(D\left(\mathcal{A}%
\right)\cap \mathcal{P}_\alpha(X)\right)\subset \mathcal{P}_{\alpha }(X)$.
\end{itemize}

Next we set $X^{\alpha }=\mathcal{P}_{\alpha }(X)$ for $\alpha =s,c,u$ and
the following straightforward lemma holds true:

\begin{lemma}
The following holds true:

\begin{itemize}
\item[(i)] The part $\mathcal{A}_{s}$ of $\mathcal{A}$ in $X^{s}$ satisfies $%
D\left(\mathcal{A}_{s}\right)=X^s$ and $r\left( \mathcal{A}_{s}\right) \leq
e^{\rho }$. We furthermore have for each $u\in X^{s}$ and each $(n,k)\in 
\mathbb{N}\times \mathbb{Z}$: 
\begin{equation*}
\left( \mathcal{A}_{s}^{n}u\right) _{k}=U_{\mathbf{A}}^{s}\left(
k,k-n\right) \Pi _{k-n}^{s}u_{k-n}.
\end{equation*}

\item[(ii)] The part $\mathcal{A}_{u}$ of $\mathcal{A}$ in $X^{u}$ satisfies 
$0\in \rho\left(\mathcal{A}_{u}\right)$ and satisfies $r\left( \mathcal{A}%
_{u}^{-1}\right) \leq e^{-\rho }$. We furthermore have for each $u\in X^{u}$
and each $(n,k)\in \mathbb{N}\times \mathbb{Z}$: 
\begin{equation*}
\left( \mathcal{A}_{u}^{-n}u\right) _{k}=U_{\mathbf{A}}^{u}(k,k+n)\Pi
_{k+n}^{u}u_{k+n}
\end{equation*}

\item[(iii)] The part $\mathcal{A}_{c}$ of $\mathcal{A}$ in $X^{c}$
satisfies $D\left(\mathcal{A}_{c}\right)=X^c$. It is invertible on $X^c$ and
satisfies: 
\begin{equation*}
r\left( \mathcal{A}_{c}\right) \leq e^{\rho _{0}}\text{ and }r\left( 
\mathcal{A}_{c}^{-1}\right) \leq e^{\rho _{0}}.
\end{equation*}%
We furthermore have for each $u\in X^{u}$ and each $(n,k)\in \mathbb{Z}%
\times \mathbb{Z}$: 
\begin{equation*}
\left( \mathcal{A}_{c}^{n}u\right) _{k}=U_{\mathbf{A}}^{c}(k,k-n)\Pi
_{k-n}^{c}u_{k-n}
\end{equation*}
\end{itemize}
\end{lemma}

\begin{remark}
The above discussion and the above lemma imply that the closed linear
operator $\mathcal{A} $ has an exponential trichotomy according to
Definition \ref{DE1.1}.
\end{remark}

Let $\mathbf{B}=\{B_{n}\}_{n\in \mathbb{Z}}$ be a bounded sequence in $%
\mathcal{L}(Y)$. Then let us consider the bounded linear operator $\mathcal{B%
}\in \mathcal{L}(X)$ defined by 
\begin{equation*}
\left( \mathcal{B}u\right) _{k}=B_{k-1}u_{k-1},\;\;\forall k\in \mathbb{Z}%
,\;\forall u\in X.
\end{equation*}%
Then note that one has: 
\begin{equation}
\Vert \mathcal{B}\Vert _{\mathcal{L}(X)}\leq \sup_{k\in \mathbb{Z}}\Vert
B_{k}\Vert _{\mathcal{L}(Y)}.  \label{3.1}
\end{equation}%
We are now interesting in the spectral properties of $\mathcal{A}+\mathcal{B}
$ by applying Theorem \ref{TH1.4}. We fix $0<\rho _{0}<\widehat{\rho _{0}}<%
\widehat{\rho }<\rho $ and $\widehat{\kappa }>\kappa $. Using the constant $%
\delta _{0}>0$ provided by Theorem \ref{TH1.4}, we fix a bounded sequence $%
\mathbf{B}=\{B_{n}\}_{n\in \mathbb{Z}}$ in $\mathcal{L}(Y)$ such that 
\begin{equation*}
\sup_{n\in \mathbb{Z}}\Vert B_{n}\Vert _{\mathcal{L}(Y)}\leq \frac{\delta
_{0}^{2}}{\kappa +\delta _{0}}.
\end{equation*}%
In view of (\ref{3.1}), Theorem \ref{TH1.4} applies to the perturbation
problem $\mathcal{A}+\mathcal{B}$ and operator $\left( \mathcal{A}+\mathcal{B%
}\right) $ has an exponential trichotomy with exponent $\widehat{\rho }_{0}$
and $\widehat{\rho }$ and with constant $\widehat{\kappa }.\ $ If we denote
the three corresponding projectors by $\widehat{\mathcal{P}}_{s},\widehat{%
\mathcal{P}}_{c},\widehat{\mathcal{P}}_{u}\in \mathcal{L}\left( X\right) $
and $\widehat{X}^{\alpha }=\widehat{\mathcal{P}}_{s}\left( X\right) $ we
have: 
\begin{equation}
\begin{pmatrix}
\left( \mathcal{A}+\mathcal{B}\right) _{s}^{.}\widehat{\mathcal{P}}_{s} \\ 
\left( \mathcal{A}+\mathcal{B}\right) _{u}^{-.}\widehat{\mathcal{P}}_{u} \\ 
\left( \mathcal{A}+\mathcal{B}\right) _{c}^{.}\widehat{\mathcal{P}}_{c}%
\end{pmatrix}%
=\left( I-\mathcal{J}\right) ^{-1}%
\begin{pmatrix}
\mathcal{A}_{s}^{.}\mathcal{P}_{s} \\ 
\mathcal{A}_{u}^{-.}\mathcal{P}_{u} \\ 
\mathcal{A}_{c}^{.}\mathcal{P}_{c}%
\end{pmatrix}%
,  \label{3.2}
\end{equation}%
wherein the bounded linear operator $\mathcal{J}$ acting on the Banach space 
$\mathcal{X}:=\mathbb{L}_{-\widehat{\rho }}(\mathbb{N},\mathcal{L}(X))\times 
\mathbb{L}_{-\widehat{\rho }}(\mathbb{N},\mathcal{L}(X))\times \mathbb{L}_{%
\widehat{\rho }_{0}}(\mathbb{Z},\mathcal{L}(X))$ is defined in (\ref{2.49}).
One furthermore has the following estimates 
\begin{equation}
\left\Vert \left( \mathcal{A}+\mathcal{B}\right) _{c}^{n}\widehat{\mathcal{P}%
}_{c}\right\Vert _{\mathcal{L}\left( X\right) }\leq \widehat{\kappa }e^{%
\widehat{\rho }_{0}\left\vert n\right\vert },\forall n\in \mathbb{Z},
\label{3.3}
\end{equation}%
\begin{equation}
\left\Vert \left( \mathcal{A}+\mathcal{B}\right) _{s}^{n}\widehat{\mathcal{P}%
}_{s}\right\Vert _{\mathcal{L}\left( X\right) }\leq \widehat{\kappa }e^{-%
\widehat{\rho }n},\forall n\in \mathbb{N},  \label{3.4}
\end{equation}%
and%
\begin{equation}
\left\Vert \left( \mathcal{A}+\mathcal{B}\right) _{u}^{-n}\widehat{\mathcal{P%
}}_{u}\right\Vert _{\mathcal{L}\left( X\right) }\leq \widehat{\kappa }e^{-%
\widehat{\rho }n},\forall n\in \mathbb{N},  \label{3.5}
\end{equation}%
as well as the following estimates for each $n\in \mathbb{N},$%
\begin{equation}
\left\Vert \left( \mathcal{A}+\mathcal{B}\right) _{s}^{n}\widehat{\mathcal{P}%
}_{s}-\mathcal{A}_{s}^{n}\mathcal{P}_{s}\right\Vert _{\mathcal{L}\left(
X\right) }\leq \frac{\kappa \delta }{\delta _{0}-\delta }e^{-\widehat{\rho }%
n},  \label{3.6}
\end{equation}%
\begin{equation}
\left\Vert \left( \mathcal{A}+\mathcal{B}\right) _{u}^{-n}\widehat{\mathcal{P%
}}_{u}-\mathcal{A}_{u}^{-n}\mathcal{P}_{u}\right\Vert _{\mathcal{L}\left(
X\right) }\leq \frac{\kappa \delta }{\delta _{0}-\delta }e^{-\widehat{\rho }%
n},  \label{3.7}
\end{equation}%
and for each $n\in \mathbb{Z}$%
\begin{equation}
\left\Vert \left( \mathcal{A}+\mathcal{B}\right) _{c}^{n}\widehat{\mathcal{P}%
}_{c}-\mathcal{A}_{c}^{n}\mathcal{P}_{c}\right\Vert _{\mathcal{L}\left(
X\right) }\leq \frac{\kappa \delta }{\delta _{0}-\delta }e^{\widehat{\rho }%
_{0}\left\vert n\right\vert }.  \label{3.8}
\end{equation}

In order to prove our perturbation result, namely Theorem \ref{TH1.8}, we
will show that the perturbed projectors exhibit a suitable structure
inherited from the one of the shift operators $\mathcal{A}$ and $\mathcal{B}$
that reads as 
\begin{equation*}
\left( \widehat{\mathcal{P}}_{\alpha }u\right) _{k}=\widehat{\Pi }%
_{k}^{\alpha }u_{k},\;\;\forall u\in X,\;\alpha =s,c,u,
\end{equation*}%
and wherein for each $k\in \mathbb{Z}$ and $\alpha =s,c,u$, $\widehat{\Pi }%
_{k}^{\alpha }$ denotes a projector of $Y$. To do so, let us introduce for
each $p\in \mathbb{Z}$ the linear bounded operator $\mathcal{D}_{p}\in 
\mathcal{L}(X)$ defined for each $u\in X$ and $k\in \mathbb{Z}$ by 
\begin{equation*}
\left( \mathcal{D}_{p}u\right) _{k}=%
\begin{cases}
u_{p}\text{ if }k=p \\ 
0\text{ if }k\neq p%
\end{cases}%
\end{equation*}%
Together with this notation, let us notice that for each $p\in \mathbb{Z}$
the following commutativity properties hold true: 
\begin{equation}
\begin{split}
& \mathcal{D}_{p}\mathcal{A}_{s}^{n}\mathcal{P}_{s}=\mathcal{A}_{s}^{n}%
\mathcal{P}_{s}\mathcal{D}_{p-n},\;\forall n\geq 0,\;\forall p\in \mathbb{Z},
\\
& \mathcal{D}_{p}\mathcal{A}_{u}^{-n}\mathcal{P}_{u}=\mathcal{A}_{u}^{-n}%
\mathcal{P}_{u}\mathcal{D}_{p+n},\;\forall n\geq 0,\;\forall p\in \mathbb{Z},
\\
& \mathcal{D}_{p}\mathcal{A}_{c}^{n}\mathcal{P}_{u}=\mathcal{A}_{c}^{n}%
\mathcal{P}_{u}\mathcal{D}_{p-n},\;\forall n\in \mathbb{Z},\;\forall p\in 
\mathbb{Z}.
\end{split}
\label{3.9}
\end{equation}%
One may also notice that $\mathcal{B}$ satisfies: 
\begin{equation}
\mathcal{D}_{p}\mathcal{B}=\mathcal{B}\mathcal{D}_{p-1},\;\forall p\in 
\mathbb{Z}.  \label{3.10}
\end{equation}%
If one considers the closed subspace $\mathcal{Z}\subset \mathcal{X}$
defined by 
\begin{equation*}
\mathcal{Z}=\left\{ 
\begin{pmatrix}
E^{s} \\ 
E^{u} \\ 
E^{c}%
\end{pmatrix}%
\in \mathcal{X}:\;%
\begin{pmatrix}
\mathcal{D}_{p}E_{.}^{s}-E_{.}^{s}\mathcal{D}_{p-.} \\ 
\mathcal{D}_{p}E_{.}^{u}-E_{.}^{u}\mathcal{D}_{p+.} \\ 
\mathcal{D}_{p}E_{.}^{c}-E_{.}^{c}\mathcal{D}_{p-.}%
\end{pmatrix}%
=0_{\mathcal{X}},\;\;\forall p\in \mathbb{Z}\right\} ,
\end{equation*}%
then we claim that

\begin{claim}
\label{CL3.3} The linear bounded operator $\mathcal{J}:\mathcal{X}%
\rightarrow \mathcal{X}$ satisfies $\mathcal{J}\mathcal{Z}\subset \mathcal{Z}
$.
\end{claim}

We postpone the proof of this claim and complete the proof of Theorem \ref%
{TH1.8}.\newline
Using the above claim and recalling that $\mathcal{Z}$ is a closed subspace
of $\mathcal{X}$ lead us to 
\begin{equation*}
\left( I-\mathcal{J}\right) ^{-1}\mathcal{Z}\subset \mathcal{Z}.
\end{equation*}%
Indeed since $\Vert \mathcal{J}\Vert _{\mathcal{L}(\mathcal{X})}<1$ then $%
\left( I-\mathcal{J}\right) ^{-1}=\sum_{k=0}^{\infty }\mathcal{J}^{k}$.
Hence due to (\ref{3.9}) one obtains that 
\begin{equation*}
\left(\mathcal{A}_{s}^{.}\mathcal{P}_{s}, \mathcal{A}_{u}^{-.}\mathcal{P}%
_{u}, \mathcal{A}_{c}^{.}\mathcal{P}_{c} \right)^T \in \mathcal{Z},
\end{equation*}%
while (\ref{3.2}) ensures that 
\begin{equation*}
\begin{pmatrix}
\left( \mathcal{A}+\mathcal{B}\right) _{s}^{.}\widehat{\mathcal{P}}_{s} \\ 
\left( \mathcal{A}+\mathcal{B}\right) _{u}^{-.}\widehat{\mathcal{P}}_{u} \\ 
\left( \mathcal{A}+\mathcal{B}\right) _{c}^{.}\widehat{\mathcal{P}}_{c}%
\end{pmatrix}%
\in \mathcal{Z}.
\end{equation*}%
The above statement completes the proof of Theorem \ref{TH1.8}. Indeed let
us first notice that the above result implies that for each $p\in \mathbb{Z}$
and $\alpha =s,c,u$, 
\begin{equation*}
\mathcal{D}_{p}\widehat{\mathcal{P}}_{\alpha }=\widehat{\mathcal{P}}_{\alpha
}\mathcal{D}_{p}.
\end{equation*}%
This means that for each $p\in \mathbb{Z}$ there exists three projectors $%
\widehat{\Pi }_{p}^{\alpha }\in \mathcal{L}(Y)$ for $\alpha =s,c,u$ such
that: 
\begin{equation*}
\left( \widehat{\mathcal{P}}_{\alpha }u\right) _{p}=\widehat{\Pi }%
_{p}^{\alpha }u_{p}\text{ for all $p\in \mathbb{Z}$, $u\in X$ and $\alpha
=s,c,u$}.
\end{equation*}%
Note that the properties $\widehat{\mathcal{P}}_{\alpha }\widehat{\mathcal{P}%
}_{\beta }=0$ for $\alpha \neq \beta $ and $\widehat{\mathcal{P}}_{s}+%
\widehat{\mathcal{P}}_{c}+\widehat{\mathcal{P}}_{u}=I_{X}$ directly re-write
as for each $p\in \mathbb{Z}$: 
\begin{equation*}
\widehat{\Pi }_{p}^{\alpha }\widehat{\Pi }_{p}^{\beta }=0\text{ for $\alpha
\neq \beta $ and }\widehat{\Pi }_{p}^{s}+\widehat{\Pi }_{p}^{c}+\widehat{\Pi 
}_{p}^{u}=I_{Y}.
\end{equation*}%
It remains to check that $\mathbf{A+B}$ has an exponential trichotomy with
constant $\widehat{\kappa }$, exponents $\widehat{\rho }_{0}<\widehat{\rho }$
and associated to the projectors $\left\{ \widehat{\Pi }_{k}^{\alpha
}\right\} _{k\in \mathbb{Z}}$ with $\alpha =s,c,u$.\newline
\textbf{Property $(ii)$ of Definition \ref{DE1.6}:}\newline
For each $k\in \mathbb{Z}$ and $\alpha =s,c,u$ one has for each $u\in D(%
\mathcal{A})$: 
\begin{equation*}
\left[ \widehat{\mathcal{P}}_{\alpha }\left( \mathcal{A}+\mathcal{B}\right) u%
\right] _{k}=\widehat{\Pi }_{k}^{\alpha }\left[ \left( \mathcal{A}+\mathcal{B%
}\right)u\right] _{k}=\widehat{\Pi }_{k}^{\alpha }U_{\mathbf{A+B}%
}(k,k-1)u_{k-1}.
\end{equation*}%
Since for each $u\in D(\mathcal{A})$ one has $\widehat{\mathcal{P}}_{\alpha
}\left( \mathcal{A}+\mathcal{B}\right)u =\left( \mathcal{A}+\mathcal{B}%
\right)\widehat{\mathcal{P}}_{\alpha }u$ and for each $k\in \mathbb{Z}$ and
each $u\in Y$ the sequences $\mathbf{u}^k=\{u_p\}_{p\in\mathbb{Z}}$ defined
by $u_p=0$ for $p\neq k-1$ and $u_{k-1}=u$ belongs to $D\left(\mathcal{A}%
\right)$, one obtains: 
\begin{equation*}
\widehat{\Pi }_{k}^{\alpha }U_{\mathbf{A+B}}(k,k-1)u=U_{\mathbf{A+B}}(k,k-1)%
\left[ \widehat{\mathcal{P}}_{\alpha }\mathbf{u}^k\right] _{k-1}=U_{\mathbf{%
A+B}}(k,k-1)\widehat{\Pi }_{k-1}^{\alpha }u.
\end{equation*}%
As a consequence one gets that for each $k\in\mathbb{Z}$ and $\alpha=s,c,u$: 
\begin{equation*}
\widehat{\Pi }_{k}^{\alpha }U_{\mathbf{A+B}}(k,k-1)=U_{\mathbf{A+B}}(k,k-1)%
\widehat{\Pi }_{k-1}^{\alpha }.
\end{equation*}
This proves statement $(ii)$.\newline
\textbf{Proof of $(iii)$ in Definition \ref{DE1.6}:}\newline
Let us set for each $k\in \mathbb{Z}$ the subspaces $\widehat{Y}_{k}^{\alpha
}=\widehat{\Pi }_{k}^{\alpha }(Y)$. Recall that $0\in \rho\left(\left( 
\mathcal{A}+\mathcal{B}\right)_u\right)$ . Hence for each $v\in \widehat{X}%
_{u}$ there exists a unique $u\in D\left(\mathcal{A}\right)\cap\widehat{X}%
_{u}$ such that $\left( \mathcal{A}+\mathcal{B}\right)u=v$. This re-writes
as for each $k\in \mathbb{Z}$: 
\begin{equation*}
U_{\mathbf{A+B}}(k,k-1)u_{k-1}=v_{k}.
\end{equation*}%
This proves that for each $k\in \mathbb{Z}$ the linear operator $U_{\mathbf{%
A+B}}(k,k-1)$ is invertible from $\widehat{Y}_{k-1}^{u}$ onto $\widehat{Y}%
_{k}^{u}$. Due to composition argument for each $k\in\mathbb{Z}$ and $n\geq
1 $, the linear operator $U_{\mathbf{A+B}}(k,k-n)$ is invertible from $%
\widehat{Y}_{k-n}^{u}$ onto $\widehat{Y}_{k}^{u}$. Furthermore one has for
each $n\geq 0$ and $p\in \mathbb{Z}$: 
\begin{equation*}
\left[ \left( \mathcal{A}+\mathcal{B}\right) ^{-n}\widehat{\mathcal{P}}_{u}v%
\right] _{p}=U_{\mathbf{A+B}}(p,p+n)\widehat{\Pi }_{p+n}^{u}v_{p+n}.
\end{equation*}%
The same arguments hold true for the central part and one obtains that for
each $k\in \mathbb{Z}$ and $n\geq 0$ the linear operator $U_{\mathbf{A+B}%
}(k,k-n)$ is invertible from $\widehat{Y}_{k-n}^{c}$ onto $\widehat{Y}%
_{k}^{c}$. Furthermore one has for each $n\in \mathbb{Z}$ and $p\in \mathbb{Z%
}$: 
\begin{equation*}
\left[ \left( \mathcal{A}+\mathcal{B}\right) ^{n}\widehat{\mathcal{P}}_{c}v%
\right] _{p}=U_{\mathbf{A+B}}(p,p-n)\widehat{\Pi }_{p-n}^{c}v_{p-n}.
\end{equation*}%
This proves that $(iii)$ is true.\newline
\textbf{Proof of $(iv)$ in Definition \ref{DE1.6}:}\newline
The proof of the growth estimates directly follow from the trichotomy
estimates for $\left( \mathcal{A}+\mathcal{B}\right) _{\alpha }^{.}$
recalled in (\ref{3.3})-(\ref{3.5}).

Finally the perturbed estimates for projected evolution semiflow stated in
Theorem \ref{TH1.8} directly follows from (\ref{3.6})-(\ref{3.8}). This
completes the proof of the result.

To complete the proof of Theorem \ref{TH1.8} it remains to prove Claim \ref%
{CL3.3}.\newline

\begin{proof}[Proof of Claim \protect\ref{CL3.3}]
Let $\left( E^{s},E^{u},E^{c}\right) ^{T}\in \mathcal{X}$ be given. Let us
set $\left( F^{s},F^{u},F^{c}\right) ^{T}=\mathcal{J}\left(
E^{s},E^{u},E^{c}\right) ^{T}$. Then according to the definition of $%
\mathcal{J}$ (see (\ref{2.49}) and (\ref{2.41})-(\ref{2.43})) for each $%
n\geq 0$ one has 
\begin{equation*}
\begin{split}
F_{n}^{s}=& \mathcal{A}_{s}^{n}\mathcal{P}_{s}-\sum_{m=0}^{+\infty }\mathcal{%
A}_{s}^{m+n}\mathcal{P}_{s}\mathcal{B}\left[ E_{m+1}^{u}+E_{-m-1}^{c}\right]
\\
& +\sum_{m=0}^{n-1}\mathcal{A}_{s}^{n-m-1}\mathcal{P}_{s}\mathcal{B}%
E_{m}^{s}-\sum_{m=0}^{+\infty }\left[ \mathcal{A}_{u}^{-m-1}\mathcal{P}_{u}+%
\mathcal{A}_{c}^{-m-1}\mathcal{P}_{c}\right] \mathcal{B}E_{n+m}^{s}.
\end{split}%
\end{equation*}%
Recalling (\ref{3.9}) and (\ref{3.10}) one directly checks that for each $%
n\geq 0$ and $p\in \mathbb{Z}$: $\mathcal{D}_{p}F_{n}^{s}=F_{n}^{s}\mathcal{D%
}_{p-n}.$ Using the formula described in (\ref{2.42}) and (\ref{2.43}) one
may directly check the claim.
\end{proof}


\begin{thebibliography}{99}
\bibitem{Barreira-Valls-2007} L. Barreira and C. Valls, Smooth center
manifolds for nonuniformly partially hyperbolic trajectories, \textit{%
Journal of Differential Equations} \textbf{237} (2007), 307-342.

\bibitem{Barreira-Valls-2008} L. Barreira and C. Valls, \textit{Stability of
Nonautonomous Differential Equations,} Lecture Notes in Mathematics, vol.
2002 (2008).

\bibitem{Barreira-Valls-2009} L. Barreira and C. Valls, Robustness of
nonuniform exponential trichotomies in Banach spaces, \textit{J. Math. Anal.
Appl. }\textbf{351} (2009) 373-381.

\bibitem{Barreira-Valls-2010} L. Barreira and C. Valls, Lyapunov sequences
for exponential trichotomies, \textit{Nonlinear Anal.} \textbf{72} (2010)
192-203.

\bibitem{Bates-Lu-Zeng-1998} P. W. Bates, K. Lu and C. Zeng, Existence and
persistence of invariant manifolds for semiflows in Banach space, \textit{%
Mem. Amer. Math. Soc.} \textbf{135} (1998), No. 645.

\bibitem{Bates-Lu-Zeng-2008} P. W. Bates, K. Lu and C. Zeng, Approximately
invariant manifolds and global dynamics of spike states, \textit{Invent.
Math.} \textbf{174} (2008) 355-433.

\bibitem{Chicone-Latushkin-1997} C. Chicone and Y. Latushkin, Center
Manifolds for Infinite Dimensional Nonautonomous Differential Equations, 
\textit{Journal of Differential Equations} \textbf{141} (1997) 356-399.\ 

\bibitem{Chicone-Latushkin-1999} C. Chicone, Y. Latushkin, \textit{Evolution
Semigroups in Dynamical Systems and Differential Equations}, Math. Surveys
Monogr., vol. 70, Amer. Math. Soc., Providence, RI, 1999.

\bibitem{Chow-Leiva-1995} S.-N. Chow, H. Leiva, Existence and roughness of
the exponential dichotomy for skew-product semiflow in Banach spaces, 
\textit{Journal of Differential Equations} \textbf{120} (1995) 429-477.

\bibitem{Chow-Leiva-1996} S.-N. Chow, H. Leiva, Two definitions of
exponential dichotomy for skew-product semiflow in Banach spaces, \textit{%
Proc. AMS} \textbf{124} (1996) 1071-1081.

\bibitem{Chow-Li-Wang-1994} S.-N. Chow, C. Li and D. Wang, \textit{Normal
Forms and Bifurcation of Planar Vector Fields,} Cambridge Univ. Press,
Cambridge, 1994.

\bibitem{Chow-Liu-Yi-2000(a)} S.-N. Chow, W. Liu, Y. Yi, Center manifolds
for invariant sets, \textit{Journal of Differential Equations} \textbf{168}
(2000) 355-385.

\bibitem{Chow-Liu-Yi-2000(b)} S.-N. Chow, W. Liu, Y. Yi, Center manifolds
for smooth invariant manifolds, \textit{Trans. Amer. Math. Soc.} \textbf{352}
(2000) 5179-5211.

\bibitem{Coppel} W.A. Coppel, Dichotomies in Stability Theory, Lecture Notes
in Math., Springer-Verlag, Berlin, 1978.

\bibitem{DeLaubenfels} R. DeLaubenfels, \textit{Existence families,
functional calculi and evolution equations}, Lecture Notes in Math., vol.
1570, Springer-Verlag, Berlin, 1994.

\bibitem{Hale} J. K. Hale, X. B. Lin, Heteroclinic Orbits for Retarded
Functional Differential Equations, \textit{Journal of differential equation} 
\textbf{65} (1986) 175-202.

\bibitem{Henry} D. Henry, Geometric theory of semilinear parabolic
equations, Springer, 1981-348 pages.

\bibitem{Magal-Ruan} P. Magal and S. Ruan, Center manifold theorem for
semilinear equations with non-dense domain and applications on Hopf
bifurcation in age structured models, Mem. Amer. Math. Soc. \textbf{202}
(2009), no. 951.

\bibitem{Latushkin-Schnaubelt} Y. Latushkin and R. Schnaubelt, Evolution
semigroups, translation algebras, and exponential dichotomy of cocycles, 
\textit{Journal of Differential Equations} \textbf{159} (1999) 321-369.

\bibitem{Lian-Lu} Z. Lian, K. Lu, Lyapunov exponents and invariant manifolds
for random dynamical systems in a Banach space, Mem. Amer. Math. Soc. 
\textbf{206} (2010), no. 967.

\bibitem{Megan-Stoica} M. Megan and C. Stoica, On Uniform Exponential
Trichotomy of Evolution Operators in Banach Spaces, \textit{Integr. equ.
oper. theory} \textbf{60} (2008) 499-506.

\bibitem{Palmer-1987} K.J. Palmer, A perturbation theorem for exponential
dichotomies, \textit{Proc. Roy. Soc. Edinburgh A} \textbf{106} (1987) 25-37.

\bibitem{Palmer-2011} K. J. Palmer, A finite-time condition for exponential
dichotomy, \textit{Journal of Difference Equations and Applications}, 
\textbf{17} (2011) 221-234.

\bibitem{Perron} O. Perron, Die Stabilitatsfrage bei Differential
gleichungen \textit{Math. Z.} \textbf{32} (1930), 703-728.

\bibitem{Pliss-Sell} V. Pliss, G. Sell, Robustness of exponential
dichotomies in infinite-dimensional dynamical systems, \textit{J. Dynam.
Differential Equations} \textbf{11} (1999) 471-513.

\bibitem{Popescu} L. Popescu, Exponential dichotomy roughness on Banach
spaces, \textit{J. Math. Anal. Appl.} \textbf{314} (2006) 436-454.

\bibitem{Pot(a)} C, P\"{o}tzsche, Exponential dichotomies of linear dynamic
equations on measure chains under slowly varying coefficients, \textit{J.
Math. Anal. Appl.} \textbf{289} (2004) 317-335.

\bibitem{Potzsche} C, P\"{o}tzsche, \textit{Geometric theory of discrete
nonautonomous dynamical systems,} vol. 2002 of Lecture Notes in Mathematics,
Springer-Verlag, Berlin, 2010.

\bibitem{Pot(b)} C, P\"{o}tzsche, Smooth roughness of exponential
dichotomies, revisited, Preprint.

\bibitem{Sacker-Sell-1974} R. J. Sacker and G. R. Sell, Existence of
dichotomies and invariant splittings for linear differential systems, I, 
\textit{Journal of Differential Equations} \textbf{22} (1974) 497-522.

\bibitem{Sacker-Sell-1976} R. J. Sacker and G. R. Sell, Existence of
dichotomies and invariant splittings for linear differential systems, III, 
\textit{Journal of Differential Equations} \textbf{22} (1976) 497-522.

\bibitem{Sacker-Sell-1978} R.J. Sacker and G.R. Sell, A spectral theory for
linear differential systems, \textit{Journal of Differential Equations} 
\textbf{27} (1978) 320-358.

\bibitem{Sacker-Sell-1994} R. J. Sacker and G. R. Sell, Dichotomies for
linear evolutionary equations in Banach spaces, \textit{Journal of
Differential Equations} \textbf{113} (1994) 17-67.

\bibitem{Sasu} B. Sasu and A. L. Sasu, Exponential trichotomy and
p-admissibility for evolution families on the real line, \textit{%
Mathematische Zeitschrift} \textbf{253} (2006) 515-536.

\bibitem{Seydi} O.\ Seydi, Infinite dimensional singularly perturbed
epidemic models: Theory and applications.\ \textit{PhD\ thesis University of
Bordeaux} 2013.\ 

\bibitem{Vanderbauwhede} A. Vanderbauwhede, Center manifold, normal forms
and elementary bifurcations, \textit{Dynamics Reported}, ed. by U.
Kirchgraber and H. O. Walther, Vol. 2, John Wiley \& Sons, 1989, 89-169.

\bibitem{Zhou-Lu-Zhang} L. Zhou, K. Lu and W. Zhang, Roughness of tempered
exponential dichotomies for infinite-dimensional random difference
equations, \textit{Journal of Differential Equations}, \textbf{254} (2013)
4024-4046.
\end{thebibliography}
\end{document}